\documentclass[reqno,preprint,aos,authoryear]{imsart}

\RequirePackage{amsthm,amsmath,amssymb,natbib}
\RequirePackage[OT1]{fontenc}
\RequirePackage[colorlinks,citecolor=blue,urlcolor=blue]{hyperref}
\usepackage{booktabs}
\usepackage{nonfloat}
\usepackage {anysize,amsmath,amssymb,mathtools}
\usepackage{pgf}

\startlocaldefs
\numberwithin{equation}{section}
\theoremstyle{plain}
\newtheorem{lemma}{Lemma}[section]
\newtheorem{assumption}{Assumption}[section]

\newtheorem{theorem}{Theorem}[section]
\newtheorem{corollary}{Corollary}[section]
\newtheorem{definition}{Definition}[section]

\newcommand{\indic}{\textbf{1}}
\newcommand{\inv}{{-1}}
\newcommand{\tepsilon}{\tilde{\epsilon}}
\def\journal@name{}
\endlocaldefs

\allowdisplaybreaks

\newcommand{\norm}[1]{\left\Vert#1\right\Vert}

\marginsize {1in} {1in} {1in} {1in}

\begin{document}                                                                                                                                                                     

\begin{frontmatter}
\title{Adaptive Bayesian Estimation of Mixed Discrete-Continuous Distributions under Smoothness and Sparsity
%\protect\thanksref{T1}
}
\runtitle{
Adaptive Bayesian Estimation of Mixed Distributions}
\thankstext{T1}{First version: December 2017,  current version: \today.}
\thankstext{T2}{We thank participants of Harvard-MIT econometrics workshop, OBayes 2017,  CIREQ 2018, and SBIES 2018 for helpful comments.}

\begin{aug}
\author{\fnms{Andriy} \snm{Norets}\thanksref{t1}
%\thanksref{t1}
\ead[label=e1]{andriy\_norets@brown.edu}
%\ead[label=u1,url]{http://www.econ.uiuc.edu/\~{ }anorets}
}
\and
\author{\fnms{Justinas} \snm{Pelenis}\thanksref{t2}
\ead[label=e2]{pelenis@ihs.ac.at}
%\ead[label=e2]{jpelenis@gmail.com}
 %\ead[label=u2,url]{http://www.princeton.edu/\~{ }jpelenis} 
}

\thankstext{t1}{Associate Professor, Department of Economics, Brown University
}

\thankstext{t2}{Assistant Professor, Vienna Institute for Advanced Studies
}

\runauthor{A. Norets and J. Pelenis}

\affiliation{Brown University and Vienna Institute for Advanced Studies}

\address{Economics Department, \\
Brown University,
Providence, RI 02912
\\
\printead{e1}
%\\
%\phantom{E-mail:\ }
%\printead*{e2}
}
\address{  Institute for Advanced Studies Vienna,
\\ Josefstaedter Strasse 39,
\\ Vienna 1080, Austria 
\printead{e2}
%\\
%\phantom{E-mail:\ }
%\printead*{e2}
}

\end{aug}

\begin{abstract}

We consider nonparametric estimation of a mixed discrete-continuous distribution 
under anisotropic smoothness conditions and
possibly increasing number of support points for the discrete part of the distribution.
For these settings, we derive lower bounds on the estimation rates in the total variation distance.  Next, we consider a nonparametric mixture of normals model that uses continuous latent variables for the discrete part of the observations.  We show that the posterior in this model contracts at rates that are equal to the derived lower bounds up to a log factor.  
Thus, Bayesian mixture of normals models can be used for optimal adaptive estimation of mixed discrete-continuous distributions.

\end{abstract}

%\begin{keyword}[class=AMS]
%\kwd[Primary ]{62G07}
%\kwd[; secondary ]{62G20.}
%\end{keyword}

\begin{keyword}
\kwd{Bayesian nonparametrics, adaptive rates, minimax rates, posterior contraction, discrete-continuous distribution, mixed scale, mixtures of normal distributions, latent variables.}
%\kwd{}
\end{keyword}

% JEL codes
% C11	Bayesian Analysis: General
% C14	Semiparametric and Nonparametric Methods: General

\end{frontmatter}

\section{Introduction}

Mixture models have proven to be very useful for Bayesian nonparametric modeling of univariate and multivariate distributions of continuous variables. 
These models possess outstanding asymptotic frequentist properties:
in Bayesian nonparametric estimation of smooth densities the posterior in these models contracts at optimal adaptive rates up to a log factor 
(\cite{Rousseau:10}, \cite{KruijerRousseauVaart:09}, \cite*{ShenTokdarGhosal2013}).
Tractable Markov chain Monte Carlo (MCMC) algorithms for exploring posterior distributions of these models are available (\cite{EscobarWest:95}, \cite{MacEachernMuller98}, \cite{Neal:00}, \cite{MillerHarrison17}, \cite{Norets2017mcmc}) and they are widely used in empirical work (see \cite*{PractNonSemiparamBayes:98}, 
\cite{ChamberlainHirano99}, \cite*{BurdaHardingHausman:08}, \cite{ChibGreenberg10}, and \cite{JensenMaheu14} among many others).

In most applications, data contain both continuous and discrete variables.  
From the computational perspective, discrete variables can be easily accommodated through 
the use of continuous latent variables in Bayesian MCMC estimation (\cite{AlbertChib_binpoly:93}, \cite{McCullochRossi:94}).  In nonparametric modelling of discrete-continuous data by mixtures,
latent variables were used by \cite{CanaleDunson:11} and \cite{NoretsPelenis2012} among others.
Some results on frequentist asymptotic properties of the posterior distribution in such models 
have also been established. \cite{NoretsPelenis2012} obtained approximation results in Kullback-Leibler distance and weak posterior consistency for mixture models with a prior on the number of mixture components. \cite{DeYoreoKottas17} establish weak posterior consistency for Dirichlet process mixtures.
In similar settings, \cite{canale2015bayesian} derived posterior contraction rates that are not optimal.
The question we address in the present paper is whether 
a mixture of normal model that uses latent variables for modeling the discrete part of the distribution
can deliver (near) optimal and adaptive posterior contraction rates for nonparametric estimation of discrete-continuous distributions.  

Our contribution has two main parts. First, we derive lower bounds on the estimation rate for mixed multivariate discrete-continuous distributions under anisotropic smoothness conditions and potentially growing support of the discrete part of the distribution.  Second, we study the posterior contraction rate 
for a mixture of normals model with a variable number of components that uses continuous latent variables for the discrete part of the observations.  We show that the posterior in this model contracts at rates that are equal to the derived lower bounds up to a log factor.  
Thus, Bayesian mixture models can be used for (up to a log factor) optimal adaptive estimation of mixed discrete-continuous distributions.  These results are obtained in a rich asymptotic framework where the multivariate discrete part of the data generating distribution can have either a large or a small number of support points and it can be either very smooth or not, and these characteristics can differ from one discrete coordinate to another.  In these settings, 
smoothing is  beneficial only for a subset of discrete variables with
a quickly growing number of support points and/or high level of smoothness. 
In a sense, this subset is automatically and correctly selected by the mixture model.
The obtained optimal posterior contraction rates are 
 adaptive since the priors we consider do not depend on the number of support points and the smoothness of the data generating process.

Our results on lower bounds have independent value outside of the literature on Bayesian mixture models and their frequentist properties.  Let us briefly review most relevant results on lower bounds and place our results in that context.
The minimax estimation rates for mixed discrete continuous distributions appear to be studied first by
\cite{Efromovich2011}.  He considers discrete variables with a fixed support and shows that the optimal rates for discrete continuous distributions are equal to the optimal nonparametric rates for 
the continuous part of the distribution.  Relaxing the assumption of the fixed support 
for the discrete part of the distribution is very desirable in nonparametric settings.
It has been commonly observed at least since \cite{AitchisonAitken76} that smoothing discrete 
data in nonparametric estimation improves results in practice.
\cite{HallTitterington87} introduced an asymptotic framework that provided a precise theoretical 
justification for improvements resulting from smoothing in the context of estimating a univariate discrete distribution with a support that can grow with the sample size.
In their setup, the support is an ordered set and the probability mass function is $\beta$-smooth (in a sense that analogs of $\beta$-order Taylor expansions hold).  They show that in their setup the minimax rate 
is the smaller one of the following two: (i) the optimal estimation rate for a continuous density with the smoothness level $\beta$, $n^{-\beta/(2\beta+1)}$, and (ii) the rate of convergence of the standard frequency estimator, $(N/n)^{1/2}$, where $N$ is the cardinality of the support and $n$ is the sample size. \cite{HallTitterington87} refer to their setup as ``Sparse Multinomial Data'' since $N$ can be larger than $n$ and this is the reason we refer to sparsity in the tile of the present paper.
\cite{Burman87} established similar results for $\beta=2$.
Subsequent literature in multivariate settings (e.g., \cite{DongSimonoff1995}, 
\cite{AertsAugustynsJanssen97statistics}) did not consider lower bounds but demonstrated that
when the support of the discrete distribution grows sufficiently fast then estimators that employ smoothing can achieve the standard nonparametric rates for $\beta$-smooth densities on $\mathbb{R}^d$, 
$n^{-\beta/(2\beta+d)}$. 

We generalize the results of \cite{HallTitterington87} on lower bounds for univariate discrete distributions to multivariate mixed discrete-continuous case and anisotropic smoothness.  
Alternatively, our results can be viewed as a generalization of results in \cite{Efromovich2011}
to settings with potentially growing supports for discrete variables.

Some details of our settings and assumptions differ from those in \cite{HallTitterington87}
and \cite{Efromovich2011} because our original motivation was in understanding the behavior of the posterior in mixture models with latent variables.
Specifically, we consider lower and upper estimation bounds in the total variation distance 
since posterior concentration in nonparametric settings is much better understood when 
the total variation distance is considered (\cite{GhosalGhoshVaart:2000}).
We also introduce a new definition of anisotropic smoothness that, on the one hand, 
accommodates an extension of techniques for deriving lower bounds from 
\cite{IbragimovHasminskii1984} and, on the other hand, lets us exploit approximation results 
for  mixtures of multivariate normal distributions developed by \cite{ShenTokdarGhosal2013}.

%\color{red}
%We show that our considered mixture model with latent variables is sufficiently adaptive for possible data generating processes within proposed anisotropic smoothness class and a growing support for the discrete part of the distribution. The behavior of the posterior and the determination of the lower bounds on estimation rates is determined by the adaptive choice on which discrete dimensions should be smoothed over and treated similarly to continuous observations and which discrete dimensions should be left unsmoothed. The decision on whether to smooth over a particular discrete dimension is performed automatically and determined based on the fineness of the discrete grid in that dimension and on the contribution of the smoothness of the underlying data generating process in that dimesnion to an ``average" smoothness level. 
%\color{black}

The rest of the paper is organized as follows.  In Section \ref{sec:notation}, we describe our framework and define notation. Section \ref{sec:main_results} presents our results on lower bounds for estimation rates.
The results on the posterior contraction rates are given in Section \ref{sec:post_rates}.
Appendix contains auxiliary results and some proofs.

\section{Preliminaries and Notation}
\label{sec:notation}

Let us denote the continuous part of observations by $x \in \mathcal{X} \subset \mathbb{R}^{d_x}$ and the discrete part by 
$y=(y_1,\ldots,y_{d_y}) \in \mathcal{Y}$, where
\begin{align*}
%\mathcal{Y}=\left\{(y_1,\ldots,y_{d_y}): \: y_k \in \left \{ \frac{1-1/2}{N_k},\frac{2-1/2}{N_k}, \ldots, \frac{N_k-1/2}{N_k} \right \}, \, k=1,\ldots,d_y\right\}
\mathcal{Y}=\prod_{j=1}^{d_y} \mathcal{Y}_j, \text{ with } 
\mathcal{Y}_j=\left \{ \frac{1-1/2}{N_j},\frac{2-1/2}{N_j}, \ldots, \frac{N_j-1/2}{N_j} \right\} ,
\end{align*}
 is a grid on $[0,1]^{d_y}$ (a product symbol $\Pi$ applied to sets hereafter denotes a Cartesian product).
The number of values that the discrete coordinates $y_j$ can take, $N_j$, can potentially grow with the sample size or stay constant.

For $y=(y_1,\ldots,y_{d_y}) \in \mathcal{Y}$, let $A_{y} = \prod_{j=1}^{d_y} A_{y_j}$, where 
%For $y_j \in \mathcal{Y}_j$, let
\begin{equation*}
A_{y_j} = \begin{cases}
(-\infty, y_j+0.5/N_j] & \text{if } y_j=0.5/N_j\\
(y_j-0.5/N_j, \infty) &\text{if } y_j=1-0.5/N_j\\
(y_j-0.5/N_j, y_j+0.5/N_j] &\text{otherwise}
\end{cases}
\end{equation*}
and let us represent the data generating density-probability mass function as 
\begin{equation}
\label{eq:p0f0repr}
p_0(y,x)= \int_{A_y} f_0(\tilde{y}, x) d\tilde{y},
\end{equation}
where $f_0$ belongs to $\mathcal{D}$, the set of probability density functions
on $\mathbb{R}^{d_x+d_y}$ with respect to the Lebesgue measure.  The representation of a mixed discrete-continuous distribution in \eqref{eq:p0f0repr} is so far without a loss of generality
since for any given $p_0$ one could always define $f_0$ using a mixture of densities with non-overlapping supports included in $A_y$, $y \in \mathcal{Y}$. 

In this paper, we consider independently identically distributed observations from $p_0$:
$(Y^n,X^n)=(Y_1,X_1,\ldots,Y_n,X_n)$. Let $P_0$, $E_0$, $P_0^n$, and $E_0^n$ denote the probability measures and expectations corresponding to $p_0$ and its product $p_0^n$.

When $N_j$'s grow with the sample size the generality of the representation in \eqref{eq:p0f0repr} can be lost when assumptions such as smoothness are imposed on $f_0$.
Nevertheless, in what follows we do impose a smoothness assumption on $f_0$.  The interpretation of this assumption is that 
the values of discrete variables can be ordered and that borrowing of information from nearby discrete points can be useful in estimation.

To get more refined results, we 
allow $N_j$'s to grow at different rates for different $j$'s.  
For the same reason, we work with anisotropic smoothness. 
Let $\mathbb{Z}_+$ denote the set of non-negative integers. For smoothness coefficients $\beta_i>0$, $i=1,\ldots,d$, $d=d_x+d_y$, 
and an envelope function $L:\mathbb{R}^{2d}\rightarrow \mathbb{R}$,
an anisotropic $(\beta_1,\ldots,\beta_d)$-Holder class, $\mathcal{C}^{\beta_1,\ldots,\beta_d,L}$, is defined as follows.

\begin{definition}
\label{def:anis_smooth}
$f \in \mathcal{C}^{\beta_1,\ldots,\beta_d,L}$ if for any $k=(k_1,\ldots,k_d) \in \mathbb{Z}_+^d$, $\sum_{i=1}^d k_i/\beta_i<1$,
mixed partial derivative of order $k$, $D^k f$, is finite and
\begin{equation}
\label{eq:smooth_def}
|D^k f (z+\Delta z)-D^k f (z) | \leq L(z, \Delta z) \sum_{j=1}^d |\Delta z_j|^{\beta_j(1-\sum_{i=1}^d k_i/\beta_i)},
\end{equation}
 %and $\Delta z$ s. t.  and 
%$1- \frac{1}{\beta_j} \leq \sum_{i=1}^d \frac{k_i}{\beta_i}<1$ iff $|\Delta z_j|>0$.
where $\Delta z_j=0$ when $\sum_{i=1}^d k_i/\beta_i + 1/\beta_j < 1 $.
\end{definition}

In this definition, a Holder condition is imposed on $D^k f$ for a coordinate $j$ when $D^k f$ cannot be differentiated with respect to $z_j$ anymore  
($\sum_{i=1}^d k_i/\beta_i<1$ but $\sum_{i=1}^d k_i/\beta_i + 1/\beta_j \geq 1 $).
This definition slightly differs from definitions available in the literature on anisotropic smoothness that we found.   
Section 13.2 in \cite{Schumaker2007splines} presents some general anisotropic smoothness definitions but restricts attention to integer smoothness coefficients.
\cite{IbragimovHasminskii1984}, and most of the literature on minimax rates under anisotropic smoothness that followed including 
\cite{BarronBirgeMassart99} and \cite{BhattacharyaPatiDunson2014}, do not restrict mixed derivatives. 
\cite{ShenTokdarGhosal2013} use $|\Delta z_j|^{\min(\beta_j - k_j,1)}$ instead of $|\Delta z_j|^{\beta_j(1-\sum k_i/\beta_i)}$ in \eqref{eq:smooth_def}. Their requirement is stronger than ours for functions with bounded support, and 
it appears too strong for our derivation of lower bounds on the estimation rate.
However, our definition is sufficiently strong to obtain a Taylor expansion with remainder terms that have the same order as those in \cite{ShenTokdarGhosal2013} 
(while the definitions that do not restrict mixed derivatives do not deliver such an expansion).

When $\beta_j=\beta$, $\forall j$ and $\sum_{i=1}^d k_i/\beta + 1/\beta \geq 1 $,
$\beta_j(1-\sum k_i/\beta_i) = \beta - \lfloor \beta \rfloor$, 
where $\lfloor \beta \rfloor$ is the largest integer that is strictly smaller than $\beta$, and we
get the standard definition of $\beta$-Holder smoothness for the isotropic case.

The envelope $L$ can be assumed to be a function of $(z,\Delta z)$ to accommodate densities with unbounded support.
We derive lower bounds on estimation rates for a constant envelope function.  
Upper bounds on posterior contraction rates are derived under more general assumptions on $L$ as in \cite{ShenTokdarGhosal2013}.

Some extra notation: for a multi-index $k=(k_1,\ldots,k_d) \in \mathbb{Z}_+^d$, $k!=\prod_{i=1}^d k_i!$, and
for $z \in \mathbb{R}^d$, $z^k=\prod_{i=1}^d z_i^{k_i}$.
The $m$-dimensional simplex is denoted by $\Delta^{m-1}$.  $I_d$ stands for the $d \times d$ identity matrix.
Let $\phi_{\mu,\sigma}(\cdot)$ and $\phi(\cdot; \mu,\sigma)$ denote a multivariate normal density with mean $\mu \in \mathbb{R}^d$ and covariance 
matrix $\sigma^2 I_d$ (or a diagonal matrix with squared elements of $\sigma$ on the diagonal
when $\sigma$ is a $d$-vector).  For $z \in \mathbb{R}^d$ and $J \subset \{1,2,\ldots,d\}$, $z_J$ denotes sub-vector $\{z_i,\, i \in J\}$.
Operator ``$\lesssim$'' denotes less or equal up to a multiplicative positive constant relation.

\section{Lower Bounds on Estimation Rates}
\label{sec:main_results}

Let $\mathcal{A}$ denote a collection of all subsets of indices for discrete coordinates  $\{1, \ldots,d_y\}$. 
For $J \in \mathcal{A}$, define $J^c=\{1,\ldots,d\}\setminus J$, 
\[
N_{J}=\prod_{i\in J} N_i, \;\;\;\;\;\;\; 
%\]
%\[
\beta_{J^c} = \left [   \sum_{i \in J^c} \beta_i^{-1}\right ]^{-1},
\]
 $N_{\emptyset}=1$, $\beta_{\emptyset} = \infty$, 
and $\beta_{\emptyset}/(2\beta_{\emptyset} + 1 )=1/2$.

For a class of probability distributions $\mathcal{P}$, $\zeta$ is said to be a lower bound on the estimation error in metric $\rho$ if
\[
\inf_{\hat{p}} \sup_{p \in \mathcal{P}} P\left(\rho(\hat{p},p)\geq \zeta \right) \geq const > 0. 
\]

We consider the following class of probability distributions: for a positive constant $L$, let 
\begin{equation}
\label{eq:classP}
\mathcal{P}=\left \{p: \: p(y,x)=\int_{A_y} f(\tilde{y}, x) d\tilde{y}, \, f \in \mathcal{C}^{\beta_1,\ldots,\beta_d,L} \cap \mathcal{D} \right\}.
\end{equation}

\begin{theorem} 
\label{th:LowerBounds}
For $\mathcal{P}$ defined in \eqref{eq:classP},
\begin{equation}
\label{eq:l_b}
\Gamma_n = \min_{J \in \mathcal{A}} \left[ \frac{N_J}{n} \right ]^{\frac{\beta_{J^c}}{2\beta_{J^c}+1}}
=\left[ \frac{N_{J_\ast}}{n} \right ]^{\frac{\beta_{J^c_\ast}}{2\beta_{J^c_\ast}+1}}
\end{equation}
multiplied by a positive constant is a lower bound on estimation error in the total variation distance.
\end{theorem}

One could recognize expression $\left[ N_J/n \right ]^{\frac{\beta_{J^c}}{2\beta_{J^c}+1}}$  in \eqref{eq:l_b}
as the standard estimation rate for a $card(J^c)$-dimensional  density with anisotropic smoothness coefficients $\{\beta_j, \, j \in J^c\}$ and the sample size $n/N_J$ (\cite{IbragimovHasminskii1984}).
One way to interpret this is 
that the density of $\{x,\tilde{y}_j,\, j \in J^c\}$ conditional on $y_J$ is $\{\beta_j, \, j \in J^c\}$-smooth and the number of observations available for its estimation (observations with the same value of $y_J$)
should be of the order $n/N_J$; also, the estimation rate for the marginal probability mass function for $y_J$ is $[N_J/n]^{1/2}$, which is at least as fast as $\left[ N_J/n \right ]^{\frac{\beta_{J^c}}{2\beta_{J^c}+1}}$.
In this interpretation, smoothing is not performed over the discrete coordinates with indices in set $J$, and the lower bound is obtained when $J$ minimizes $\left[ N_J/n \right ]^{\frac{\beta_{J^c}}{2\beta_{J^c}+1}}$.
Thus, an estimator that delivers the rate in \eqref{eq:l_b} should, in a sense, optimally choose the subset of discrete variables over which to perform smoothing.

We set up the notation and an outline of the proof of Theorem \ref{th:LowerBounds} below and delegate detailed calculations to lemmas in Appendix \ref{sec:proofs_extras}.
The proof of the theorem is based on a general theorem 
from the literature on lower bounds, which we present next in a slightly simplified form.

\begin{lemma}
\label{lm:AbstrLBTsyb}
(Theorem 2.5 in \cite{Tsybakov:08}, see also \cite{IbragimovHasminskii1977})
$\zeta$ is a lower bound on the estimation error in metric $\rho$ for a class $\mathcal{Q}$ if
there exist a positive integer $M \geq 2$ and $q_j,q_i \in \mathcal{Q}$, $0\leq j < i \leq M$ such that
$\rho(q_j,q_i) \geq 2 \zeta$, $q_j << q_0$, $j=1,\ldots,M$ and 
\begin{equation}
\label{eq:avgKLcond}
\sum_{j=1}^M KL(Q_{j}^n,Q_{0}^n)/M < \log(M)/8,
\end{equation}
where $KL$ is the Kullback-Leibler divergence and $Q_{j}^n$ is the distribution of a random sample from $q_j$.
\end{lemma}

The following standard result on bounding the number of unequal elements in binary sequences is used in our construction of $q_j$, $j=1,\ldots,M$.

\begin{lemma}
\label{lm:VarshamovGilbert}
(Varshamov-Gilbert bound, Lemma 2.9 in \cite{Tsybakov:08})
Consider the set of all binary sequences of length $\bar{m}$,
%\[
$\Omega =\left\{w=(w_1, \ldots, w_{\bar{m}}):\; w_r \in \{0, 1\} \right \} = \{0, 1\}^{\bar{m}}.$
%\]
Suppose $\bar{m} \geq 8$. Then there exists
a subset $\{w^1, \ldots, w^M\}$ of $\Omega$ such that $w^0 = (0,\ldots, 0)$,
\[
\sum_{r=1}^{\bar{m}} 1\{w^j_r\neq w^i_r\} \geq \bar{m}/8, \; \forall 0 \leq j < i \leq M, 
\]
and
\[M \geq 2^{\bar{m}/8}.
\]
\end{lemma}

To define $q_j$'s for our problem, we need some additional notation.
Let 
\[
K_0(u) = \exp\{-1/(1-u^2)\} \cdot 1\{|u|\leq 1\}.
\]
This function has bounded derivatives of all orders and it 
smoothly decreases to zero at the boundary of its support.
This type of kernel functions is usually used for constructing hypotheses for lower bounds, see Section 2.5 in \cite{Tsybakov:08}.
Since we  need to construct a smooth density 
that integrates to 1, we define 
(as illustrated in Figure \ref{fig:fun_g}) 
%\[
%g(u)=c_0 [K_0(8(u+1/4))-K_0(8(u-1/4))],
%\]
%\color{red} (Figure has been removed or should we draw a figure of it?)
\[
g(u)=c_0 [K_0(4(u+1/4))-K_0(4(u-1/4))],
\]
%\color{black}
where $c_0>0$ is a sufficiently small constant that will be specified below.
\\[\intextsep]
\begin{minipage}{\linewidth}
\centering%
\includegraphics[width=100mm]{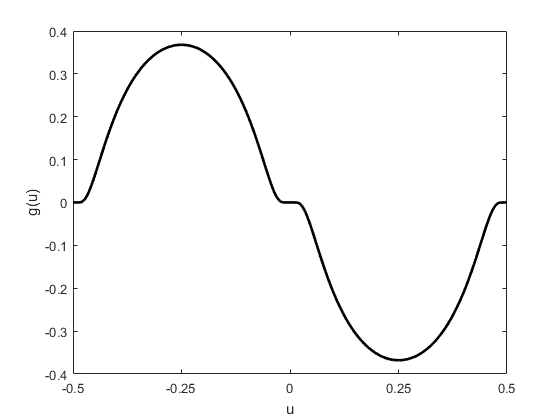}%
\figcaption{Function $g$ for $c_0=1$.}%
\label{fig:fun_g}
\end{minipage}
\\[\intextsep]
Function $g$ will be used as a kernel in construction of $q_k$'s.  Let us define the bandwidth for these kernels first.

For the continuous  coordinates, we define the bandwidth as in \cite{IbragimovHasminskii1984},
\[
h_i=\Gamma_n^{1/\beta_i},\, i \in \{d_y+1,\ldots,d\}. %J_\ast^c \text{ and } i>d_y
\]
For the discrete ones, over which smoothing is beneficial, 
we define the bandwidth as
\[
h_i=\varrho_i\cdot\Gamma_n^{1/\beta_i}=\frac{2}{N_i}\cdot R_i,\, i \in J_\ast^c \cap \{1,\ldots,d_y\},
\]
where 
$R_i=\lfloor \Gamma_n^{1/\beta_i} N_i / 2\rfloor +1$ is a positive integer and 
$\varrho_i\in(1,2]$ as shown in  %part (ii) of 
Lemma \ref{lm:b_bast_h_N}.

%For continuous coordinates and those discrete ones over which smoothing is beneficial, 
%we define the bandwidth
%\[
%h_i=8\cdot\Gamma_n^{1/\beta_i},\, i \in J_\ast^c,
%\]
%where $\Gamma_n^{1/\beta_i}$ part is standard (\cite{IbragimovHasminskii1984}) and
%constant $8$ is determined by the definition of $g$ and its use in the proofs below.
For the rest of the discrete coordinates, our innovation is to first define artificial anisotropic smoothness coefficients $\beta_i^\ast=-\log(\Gamma_n)/\log{N_i}$, $i \in J_\ast$, at which the rate in 
\eqref{eq:l_b} would have the same value whether we smooth over $y_i$ ($i \in J_\ast^c$) or not ($i \in J_\ast$).
Then, we define the bandwidth as
\[h_i=2\cdot \Gamma_n^{1/\beta_i^\ast}=2/N_i, \ i \in J_\ast.\]
To streamline the notation, we also define $\beta_i^\ast=\beta_i$ for $i \in J_\ast^c$.
%The rest of the construction of the hypotheses is an adaptation of arguments from the literature on lower bounds for anisotropic densities.

Let $m_i$ be the integer part of $h_i^{-1}$, $i=1,\ldots,d$.  
%Let us divide $[0,1]^d$ into $m=\prod_{i=1}^d m_i$ rectangles with the side lengths $(h_1,\ldots,h_d)$ and 
%centers $c^r=(c^r_1,\ldots,c^r_d)$, $r=1,\ldots,m$, so that the rectangles do not have positive volume overlap.
Let us consider $\bar{m}=\prod_{i=1}^d m_i$ adjacent rectangles in $[0,1]^d$, $B_r$, $r=1,\ldots,\bar{m}$, with the side lengths $(h_1,\ldots,h_d)$ and 
centers $c^r=(c^r_1,\ldots,c^r_d)$,  
 $c_i^r=h_i(k_{ir} - 1/2)$, $k_{ir} \in\{1,\ldots,m_i\}$.
For $z \in \mathbb{R}^d$ and $r=1,\ldots,\bar{m}$, define 
\[
g_r (z)= \Gamma_n \prod_{i=1}^d g((z_i-c^r_i)/h_i),
\]
which can be non-zero only on $B_r$.
A set of hypotheses is defined by sequences of binary weights on $g_r$'s as follows
\begin{equation}
\label{eq:q_j}
q_j(y,x) = \int_{A_y} \left [1_{[0,1]^d}(\tilde{y},x) + \sum_{r=1}^{\bar{m}} w^j_r g_r(\tilde{y},x) \right ] d\tilde{y},
%q_j(x,y) = \int_{A_y} \left [1_{[0,1]^d}(x,\tilde{y}) + \sum_{r=1}^{\bar{m}} w^j_r g_r(x,\tilde{y}) \right ] d\tilde{y},
\end{equation}
where $w^j_r \in \{0,1\}$, $j=0,\ldots,M$, and $M$ are defined in Lemma \ref{lm:VarshamovGilbert}.

The rest of the proof is delegated to lemmas in Appendix \ref{sec:proofs_extras},
which show that $q_k$ in \eqref{eq:q_j} satisfy the sufficient conditions 
from Lemma \ref{lm:AbstrLBTsyb}. 
Specifically, Lemma \ref{lm:q_tvd} derives the lower bound on the total variation distance.
Lemma \ref{lm:q_kl} verifies condition \eqref{eq:avgKLcond} 
when $\bar{m}\geq 8$.
Lemma  \ref{lm:q_smooth}, part (i) of Lemma \ref{lm:b_bast_h_N}, and the fact that $q_k$'s are defined on $[0,1]^d$ imply $q_j \in \mathcal{C}^{\beta_1,\ldots,\beta_d,L}$, $j=0,\ldots,M$.

This argument (Lemma \ref{lm:q_kl} specifically) requires $\bar{m}\geq 8$ as it relies on Lemma \ref{lm:VarshamovGilbert}.
Observe that as $n\rightarrow \infty$, $\bar{m}\geq 8$ if there are continuous variables or there are discrete variables over which smoothing is beneficial ($J_\ast^c \neq \emptyset$). 
Thus, $\bar{m}<8$ can happen only if there are no continuous variables and $N_{J_\ast}=N_1 \cdots N_d$ is bounded. This is just a problem of estimating a multinomial distribution with finite support 
and the standard results for parametric problems deliver the usual $n^{-1/2}$ rate.

Finally, note that we prove the lower bound results for a class of densities that includes densities that are in $\mathcal{C}^{\beta_1,\ldots,\beta_d,L}$ on $[0,1]^d$.
It is straightforward to modify the proof so that it works for a class of 
smooth densities on $\mathbb{R}^d$.  To accomplish this we can replace $ 1_{[0,1]^d}(\cdot) $
in \eqref{eq:q_j} with a smooth function on $\mathbb{R}^d$ that has a bounded support and is bounded away from zero on $[0,1]^d$, for example,  
\[
\prod_{i=1}^d \left[ 1_{[0,1]}(z_i) + IK_0(z_i+1)*1(z_i<0) + IK_0(2-z_i)*1(z_i>1) \right],
\]
multiplied by a normalization constant, 
where $IK_0(z_i)=\int_{-1}^{z_i} K_0(u)du / \int_{-1}^1 K_0(u)du$.
Then, proofs of Lemmas  \ref{lm:q_tvd}-\ref{lm:q_smooth} go through with minor modifications.

\pagebreak

\section{Posterior Contraction Rates for a Mixture of Normals Model}
\label{sec:post_rates}

\subsection{Model and Prior}
\label{sec:model_prior}

In this section, we consider a Bayesian model for the data generating process in \eqref{eq:p0f0repr}.  We use a mixture of normal distributions with a variable number of components for modelling the joint distribution of $(\tilde{y},x)$,
\begin{align}
\label{eq:joint_mix_def}
f(\tilde{y},x| \theta, m) &= \sum_{j=1}^m \alpha_j \phi(\tilde{y},x;\mu_j,\sigma)  \notag \\
p(y,x| \theta, m) &= \int_{A_y}  f(\tilde{y},x| \theta, m) d\tilde{y},
\end{align}
where $\theta=(\mu_j^y, \mu_j^x, \alpha_j, j=1,2,\ldots,m; \sigma)$.

We assume the following conditions on the prior $\Pi$ for $(\theta,m)$.  For positive constants $a_1, a_2,  \ldots, a_{9}$, for each $i\in \{1,\ldots,d\}$ the prior for $\sigma_i$ satisfies 
\begin{eqnarray}  
	\Pi( \sigma^{-2}_i \geq s)  &\leq& a_1 \exp \{-a_2 s^{a_3} \} \quad \text{for all sufficiently large} \, \,  s > 0 
	\label{eq:asnPrior_sigma1}\\
	\Pi( \sigma^{-2}_i < s)  &\leq& a_4 s ^{a_5}  \quad \text{for all sufficiently small}  \, \,  s > 0 \label{eq:asnPrior_sigma2}\\
	\Pi\{ s < \sigma^{-2}_i < s(1+t) \} &\geq& a_6 s^{a_7} t^{a_8} \exp \{-a_9 s^{1/2}\}, \quad s > 0, \quad t \in (0,1). 
	\label{eq:asnPrior_sigma3}
\end{eqnarray}
An example of a prior that satisfies \eqref{eq:asnPrior_sigma1}-\eqref{eq:asnPrior_sigma3} is the inverse Gamma prior 
for $\sigma_i$. 

Prior for 
$(\alpha_1,\ldots,\alpha_m)$ conditional on $m$ is Dirichlet$(a/m,\ldots,a/m)$, $a > 0$.
 Prior for the number of mixture components $m$ is
\begin{equation}
	\label{eq:asnPrior_m}
	\Pi(m=i) \propto \exp(-a_{10} i (\log i)^{\tau_1}), i =2, 3, \ldots, \quad a_{10}>0, \tau_1 \geq 0.
\end{equation}
More generally, a prior that can be bounded above and below by functions in the form of the right hand side of \eqref{eq:asnPrior_m}, possibly with different constants, would also work.

A priori, the components of $\mu_{j}$, $\mu_{j,i}$, $i=1,\dots,d$ are independent from each other, other parameters, and across $j$.
Prior density for $\mu_{j,i}$ is bounded below for some $a_{12}, \tau_2 > 0$ by 
\begin{equation}
	\label{eq:asnPrior_mu_lb}
	a_{11}\exp(-a_{12} |\mu_{j,i}|^{\tau_2} ),
\end{equation}
	and for some $a_{13}, \tau_3>0$ and all sufficiently large $\mu > 0$,
	\begin{equation}
	\label{eq:asnPrior_mu_tail_ub}
	\Pi(\mu_{j,i} \notin [-\mu,\mu]) \leq \exp(-a_{13} \mu^{\tau_3}). 
\end{equation}

\subsection{Assumptions on the Data Generating Process}
\label{sec:assumpt_dgp}

In what follows, we consider a fixed subset of discrete indices $J \in \mathcal{A}$ and show that under regularity conditions,
the posterior contraction rate is bounded above by $\left[ \frac{N_J}{n} \right ]^{\frac{\beta_{J^c}}{2\beta_{J^c}+1}}$ times a log factor.
If the regularity conditions we describe below for a fixed $J$ hold for every subset of $\mathcal{A}$, then the posterior contraction rate matches the lower bound in
\eqref{eq:l_b} up to a log factor.

Without a loss of generality, let $J=\{1,\ldots,d_J\}$, $I=\{d_{J}+1,\ldots,d_y\}$, $J^c=\{1,\ldots,d\}\setminus J$, and $d_{J^c}=card(J^c)$.
Similarly to $\mathcal{Y}$ and $A_{y}$ defined in Section \ref{sec:notation}, we define 
$\mathcal{Y}_J=\prod_{j\in J} \mathcal{Y}_j$ and $A_{y_J}= \prod_{i \in J} A_{y_i}$.
Also, let  $y_{J}=\{y_i\}_{i \in J}$, $\tilde{y}_{I}=\{\tilde{y}_i\}_{i \in I}$, $\tilde{x}=(\tilde{y}_I,x) \in \tilde{\mathcal{X}}=\mathbb{R}^{d_{J^c}}$.

To formulate the assumptions on the data generating process, we need additional notation,
\begin{align*}
f_{0J}(y_J,\tilde{x})&= \int_{A_{y_J}} f_0(\tilde{y}_J, \tilde{x}) d\tilde{y}_J,\\
\pi_{0J}(y_J)&=\int_{\tilde{\mathcal{X}}} f_{0J}(y_J,\tilde{x}) d\tilde{x}, \\
f_{0|J}(\tilde{x}|y_J)&=\frac{ f_{0J}(y_J,\tilde{x})}{\pi_{0J}(y_J)},\\
p_{0|J}(y_I,x|y_J)&= \int_{A_{y_I}} f_{0|J}(\tilde{y}_I, x|y_J) d\tilde{y}_I.
\end{align*}
Also, let $F_{0|J}$ and $E_{0|J}$ denote the conditional probability and expectation corresponding to $f_{0|J}$.
If $\pi_{0J}(y_J)=0$ for a particular $y_J$, then we can define the conditional density $f_{0|J}(\tilde{x}|y_J)$ arbitrarily.
We make the  following assumptions on the data generating process.

\begin{assumption}
\label{as:f0_subexp_tail}
There are positive finite constants $b,\bar{f}_0,\tau$ such that for any $y_J \in\mathcal{Y}_J$ and $\tilde{x} \in \tilde{\mathcal{X}}$
\begin{align}
%\label{eq:f0J_ub}
\label{eq:tail_cond2}
f_{0|J}(\tilde{x}|y_J ) \leq \bar{f}_0 \exp \left( -b ||\tilde{x}||^{\tau}\right).
\end{align}
\end{assumption}

It appears that all the papers on (near) optimal posterior contraction rates for mixtures of normal densities impose similar tail conditions on data generating densities.

\begin{assumption}
\label{as:f0_Ay_tail_unif_x}
There exists a positive and finite $\bar{y}$ such that for any
$(y_I,y_J) \in \mathcal{Y}$ and $x \in \mathcal{X}$ 
\begin{equation}
\label{eq:ass_y_I_tail_indep_x}
\int_{A_{y_I} \cap \{||\tilde{y}_I|| \leq \bar{y} \}}
f_{0|J}(\tilde{y}_I, x |y_J) d\tilde{y}_I
\geq
\int_{A_{y_I} \cap \{||\tilde{y}_I|| > \bar{y} \}}
f_{0|J}(\tilde{y}_I, x |y_J) d\tilde{y}_I.
\end{equation}
\end{assumption}

This assumption always holds for $A_{y_I} \subset [0,1]^{d_{J^c}-d_x}$.  When $A_{y_I}$ is a rectangle with at least one infinite side,
an interpretation of this assumption is that the tail probabilities for $\tilde{y}_I$ conditional on $(x,y_J)$ decline uniformly in $(x,y_J)$.
Bounded support for $\tilde{y}_I$ is a sufficient condition for this assumption.

%\begin{assumption}
%\label{as:pi0_gl_barpiN}
%There are $0<\underline{\pi}<\bar{\pi}<\infty$, such that
%for any $y_J \in\mathcal{Y}_J$, $n$, and $N_J$,
%\begin{equation}
%\label{eq:pi0J_lbub}
%\frac{\underline{\pi}}{N_J}\leq \pi_{0J}(y_{J})\leq \frac{\bar{\pi}}{ N_J}.
%\end{equation}
%\end{assumption}
%A sufficient condition for this assumption is that the marginal density for $\tilde{y}_J$ is bounded above and away from zero on $[0,1]^{d-d_{J^c}}$.

\begin{assumption}
\label{as:f0_in_CbL}
We assume that 
\begin{equation}
\label{eq:foInC}
f_{0|J} \in \mathcal{C}^{\beta_{d_J+1},\ldots,\beta_{d},L},
\end{equation}
where for some $\tau_0\geq 0$ and any $(\tilde{x}, \Delta \tilde{x}) \in \mathbb{R}^{2d_{J^c}}$
\begin{equation}
\label{eq:L_defb}
L(\tilde{x},\Delta \tilde{x}) = \tilde{L}(\tilde{x}) \exp\left\{ \tau_0 ||\Delta \tilde{x}||^2 \right\},
\end{equation}
\begin{equation}
\label{eq:tildeL_ineq}
\tilde{L}(\tilde{x}+\Delta \tilde{x}) \leq \tilde{L}(\tilde{x}) \exp\left\{ \tau_0 ||\Delta \tilde{x}||^2 \right\}.
\end{equation}
\end{assumption}

The smoothness assumption \eqref{eq:foInC} on the conditional density $f_{0|J}$ is implied by the smoothness of the joint density $f_0$ at least under boundedness away from zero assumption, 
see Lemma \ref{lm:finCb_fcondJinC}.

\begin{assumption}
\label{as:Ef0_tildeL}
There are positive finite constants $\varepsilon$ and $\bar{F}$, such that 
for any $y_J \in\mathcal{Y}_J$ and $ k =\{k_i\}_{i \in J^c} \in \mathbb{N}_0^{d_{J^c}}$, $\sum_{i \in J^c} k_i/\beta_i < 1$, 
\begin{align}
\label{eq:asnE0Dff0}
 & \int \left[\frac{|D^k f_{0|J}(\tilde{x}|y_J)|}{f_{0|J}(\tilde{x}|y_J)} \right]^{\frac{(2+\varepsilon\beta_{J^c}^{-1} d_{J^c}^{-1})}{\sum_{i \in J^c} k_i/\beta_i}}  f_{0|J}(\tilde{x}|y_J) d \tilde{x} <\bar{F}, \\ 
\label{eq:asnE0Lf0}
 & \int \left[\frac{\tilde{L}(\tilde{x})}{f_{0|J}(\tilde{x}|y_J)}\right]^{2+\varepsilon\beta_{J^c}^{-1} d_{J^c}^{-1}} f_{0|J}(\tilde{x}|y_J) d \tilde{x} < \bar{F}.
\end{align}
\end{assumption}
The envelope function and restrictions on its behaviour are mostly relevant for the case of unbounded support.  Condition \eqref{eq:asnE0Lf0} suggests that the envelope function $\tilde{L}$ should be comparable to $f_{0|J}$.

\begin{assumption}
\label{as:Nj_o_n_1nu}
For some small $\nu>0$,
\begin{equation}
\label{eq:Npoly_n}
N_J=o(n^{1-\nu}).
\end{equation}
\end{assumption}
We impose this assumption 
to exclude from consideration
the cases with 
very slow (non-polynomial) rates
%such as those resulting from $N_J/n=1/\log n$
as some parts of the proof require $\log (1/\epsilon_n)$ to be of order $\log n$.

\subsection{Posterior Contraction Rates}
\label{sec:contr_rate}

Let 
\begin{equation}
\label{eq:t0def}
t_{J0} = \begin{cases}
\frac{d_{J^c}[1 + 1/(\beta_{J^c}d_{J^c}) + 1/\tau] + \max\{\tau_1,1,\tau_2/\tau\}} {2 + 1/\beta_{J^c}} & \text{if } J^c\neq\emptyset\\
\max\{\tau_1,1\} / 2 &\text{if }J^c = \emptyset
\end{cases}
\end{equation}
where  $(\tau,\tau_1, \tau_2)$ are defined in Sections \ref{sec:model_prior}-\ref{sec:assumpt_dgp}.

\begin{theorem} \label{th:post_contr}
Suppose the assumptions from Sections \ref{sec:model_prior}-\ref{sec:assumpt_dgp} hold for a given $J \in \mathcal{A}$.
Let
\begin{equation}
\label{eq:def_eps}
\epsilon_n =  
\left[\frac{N_J}{n}\right]^{\beta_{J^c}/(2\beta_{J^c} + 1 )} (\log n)^{t_J},
\end{equation}
where  $t_J > t_{J0} + \max \{0, (1- \tau_1)/2\}$.
Suppose also $n \epsilon_n^2 \rightarrow \infty$.
Then, there exists $\bar{M}>0$ such that 
\[
\Pi\left( p: d_{TV}(p,p_0) > \bar{M} \epsilon_n | Y^n,X^n\right) \stackrel{P_0^n}{\rightarrow} 0.
\]
\end{theorem}
As in Section \ref{sec:main_results}, when $J^c = \emptyset$, $\beta_{J^c}$ can be defined to be infinity and $\beta_{J^c}/(2\beta_{J^c} + 1 )=1/2$ in \eqref{eq:def_eps}.

\begin{corollary} \label{cl:post_contr_min}
Suppose the assumptions from Sections \ref{sec:model_prior}-\ref{sec:assumpt_dgp} hold for every $J \in \mathcal{A}$.
Let
\begin{equation}
\label{eq:eps_w_min}
\epsilon_n =  \min_{J \in \mathcal{A}}
\left[\frac{N_J}{n}\right]^{\beta_{J^c}/(2\beta_{J^c} + 1 )} (\log n)^{t_{J}},
\end{equation}
where  $t_J > t_{J0} + \max \{0, (1- \tau_1)/2\}$.
Suppose also $n \epsilon_n^2 \rightarrow \infty$.
Then, there exists $\bar{M}>0$ such that 
\[
\Pi\left( p: d_{TV}(p,p_0) > \bar{M} \epsilon_n | Y^n,X^n\right) \stackrel{P_0^n}{\rightarrow} 0.
\]
\end{corollary}

Under the assumptions of the corollary, Theorem \ref{th:post_contr} delivers a valid upper bound on the posterior contraction rate for every $J \in \mathcal{A}$ including the one for which 
the minimum in \eqref{eq:eps_w_min} is attained.  Hence, the corollary is an immediate implication of  Theorem \ref{th:post_contr} whose proof is presented in the following section.

The results on lower bounds in Section \ref{sec:main_results} hold for
any class of data generating densities that includes $f_0$ satisfying the following conditions:
$f_0 \in \mathcal{C}^{\beta_1,\ldots,\beta_d,L}$, $f_0=0$ outside $[0,1]^d$, and
$\overline{f} \geq f_0\geq \underline{f}>0$, where 
$L$, $\overline{f}$, and $\underline{f}$ are finite positive constants.  
It is worth pointing out that these conditions imply 
Assumptions \ref{as:f0_subexp_tail}-\ref{as:Ef0_tildeL},
which combined with Assumption \ref{as:Nj_o_n_1nu} for $N_{J_\ast}$ and a prior specified in Section \ref{sec:model_prior} would deliver 
the sufficient conditions of Corollary \ref{cl:post_contr_min}.

\subsection{Proof of Posterior Contraction Results}
\label{sec:post_c_proof_discuss}

To prove Theorem \ref{th:post_contr}, we use
the following sufficient conditions for posterior contraction from Theorem 2.1 in \cite{GhosalVaart:01}.
Let $\epsilon_n$ and $\tilde{\epsilon}_n$ be positive sequences with 
$\tilde{\epsilon}_n \leq \epsilon_n$,  $\epsilon_n \to 0$, and $n \tilde{\epsilon}_n^2 \to \infty$, and
$c_1$, $c_2$, $c_3$, and $c_4$ be some positive constants. 
Let $\rho$ be Hellinger or total variation distance.
Suppose $\mathcal{F}_n \subset \mathcal{F}$ is a sieve with the following bound on the metric entropy $M_e( \epsilon_n, \mathcal{F}_n, \rho)$
	\begin{eqnarray} \label{eq:entropy}
	\log M_e( \epsilon_n, \mathcal{F}_n, \rho) \leq c_1 n \epsilon_n^2, 
	\end{eqnarray}
%		and for $\tilde{\epsilon}_n \leq \epsilon_n$,
	\begin{eqnarray}\label{eq:sievecomplement}
	\Pi(\mathcal{F}_n^c) \leq c_3 \exp\{ -(c_2+4)n \tilde{\epsilon}_n^2\}.
	\end{eqnarray}
	Suppose also that the prior thickness condition holds
	\begin{equation}
	\label{eq:prior_thick}
	\Pi(\mathcal{K}(p_0,\tilde{\epsilon}_n)) \geq c_4 \exp\{ -c_2 n \tilde{\epsilon}_n^2\},
	\end{equation}
	where the generalized Kullback-Leibler neighborhood $\mathcal{K}(p_0,\tilde{\epsilon}_n)$ is defined by
	\begin{align*}
\mathcal{K}(p_0,\epsilon)=\left \{p: \; \int_{\mathcal{X}} \sum_{y\in \mathcal{Y}} p_0(y,x) \log \frac{ p_0(y,x)}{p(y,x)}dx< \epsilon^2, \; \int_{\mathcal{X}} \sum_{y\in \mathcal{Y}} p_0(y,x) \left[\log\frac{ p_0(y,x)}{p(y,x)}\right]^2 dx< \epsilon^2
 \right \}.
\end{align*}
Then, there exists $\bar{M}>0$ such that 
\[
\Pi\left( p: \rho(p,p_0) > \bar{M} \epsilon_n | Y^n,X^n\right) \stackrel{P_0^n}{\rightarrow} 0.
\]

The definition of the sieve  and verification of conditions \eqref{eq:entropy} and \eqref{eq:sievecomplement} closely follow analogous results in the literature on contraction rates for mixture models
in the context of density estimation.  The details are given in Lemma \ref{lm:sieve_n} in Appendix \ref{sec:app_sieve_entropy}.
Verification of the prior thickness condition is more involved and we formulate it as a separate result in the following theorem.

\begin{theorem} \label{th:prior_thickness}
Suppose the assumptions from Sections \ref{sec:model_prior}-\ref{sec:assumpt_dgp} hold for a given $J \in \mathcal{A}$.
Let $t_J > t_{J0}$, where $t_{J0}$ is defined in \eqref{eq:t0def}, and
\begin{equation}
\label{eq:def_tilde_eps}
\tilde{\epsilon}_n =  
\left[\frac{N_J}{n}\right]^{\beta_{J^c}/(2\beta_{J^c} + 1 )} (\log n)^{t_J}.
\end{equation}
For any $C > 0$ and all sufficiently large $n$,
\begin{eqnarray}\label{eq:KL}
\Pi( \mathcal{K}(p_0,  \tilde{\epsilon}_n)) \geq \exp \{-C n \tilde{\epsilon}_n^2 \}.
\end{eqnarray}  
\end{theorem}

%Let us first outline the main steps of the proof and highlight new issues that do not arise in the proofs of analogous results for continuous densities.
%A precise argument with some details delegated to lemmas in Appendix \ref{sec:app_prior_thickness} follows.

Approximation results are key for showing the prior thickness condition \eqref{eq:KL}.
Appropriate approximation results for $f_{0J}(y_J,\tilde{x})= f_{0|J}(\tilde{x}|y_J) \pi_{0J}(y_J)$ are obtained as follows.
Based on approximation results for continuous densities by normal mixtures from \cite{ShenTokdarGhosal2013}, we obtain approximations
for $f_{0|J}(\cdot|y_J)$ for every $y_J$ in the form
\begin{equation}
\label{eq:f_J_def}
f_{|J}^\star(\tilde{x}|y_J)=\sum_{j=1}^K \alpha^\star_{j|y_J}\phi(\tilde{x};\mu^\star_{j|y_J}, 
\sigma_{J^c}^\star),
%\Sigma_x^\star).
\end{equation}  
where the parameters of the mixture will be defined precisely below.
For the discrete variables over which smoothing is not performed, $y_J$, we show that
$\pi_{0J}(y_J)$ can be appropriately approximated by
\[
\int_{A_{y_J}} \sum_{y_J^\prime} \pi_{0J}(y_J^\prime) \phi(\tilde{y}_J; y_J^\prime, \sigma_{J}^\star) d\tilde{y}_J,
\]
where $\int_{A_{y_J}} \phi(\tilde{y}_J, y_J^\prime, \sigma_{J}^\star) d\tilde{y_J}$ behaves like an indicator $1\{{y_J}=y_J^\prime\}$ for sufficiently small $\sigma_{J}^\star$.  
The following subsection presents proof details.

\subsubsection{Proof of Theorem \ref{th:prior_thickness} for $J^c\neq \emptyset$}%\footnote{We are still proofreading and revising the proof, typos and notation inconsistencies are likely.}
%First we consider the case of $J^c\neq \emptyset$. 
Define $\beta= d_{J^c} \left [   \sum_{k \in J^c} \beta_k^{-1}\right ]^{-1}$, $\beta_{\min}=\min_{j\in J^c} \beta_j$, and $\sigma_n = [\tilde{\epsilon}_n / \log (1/ \tilde{\epsilon}_n) ]^{1/\beta}$.
For $\varepsilon$ defined in \eqref{eq:asnE0Dff0}-\eqref{eq:asnE0Lf0},
$b$ and $\tau$ defined in \eqref{eq:tail_cond2},
and a sufficiently small $\delta>0$,
let $a_0 = \{ (8\beta + 4\varepsilon +8 + 8 \beta/\beta_{\min})/(b \delta)\}^{1/\tau}$,
%\footnote{ This constant $a_0$ is marginally different from the original $a_0^{STG} = \{ (8\beta + 4\varepsilon +16)/(b \delta)\}^{1/\tau}$ in \cite{ShenTokdarGhosal2013} to accommodate for anisotropic data generating process. However, all the statements are still true as $a_0$ is simply $a_0^{STG}$ multiplied by a constant.}
$a_{\sigma_n} = a_0 \{\log (1/\sigma_n) \}^{1/\tau}$, 
and $b_1 > \max \{1, 1/ 2\beta \}$ satisfying $\tilde{\epsilon}_n^{b_1}  \{  \log (1/ \tilde{\epsilon}_n) \}^{5/4} \leq \tilde{\epsilon}_n$.
Then, the proofs of Theorems 4 and 6 in \cite{ShenTokdarGhosal2013} 
imply the following two claims for each $y_J=k\in\mathcal{Y}_J$ under the assumptions of Section \ref{sec:assumpt_dgp}.

First, %by Lemma \ref{lm:EllipsoidNumber}, 
there exists a partition $\{U_{j|k}, j=1,\ldots,K\}$ of $\{\tilde{x} \in \tilde{\mathcal{X}}: ||\tilde{x}||
 \leq 2a_{\sigma_n}\}$,
 such that for $j=1,\ldots,N$,
$U_{j|k}$ is contained within an ellipsoid with
center $\mu^\star_{j|k}$ and radii $\{ \sigma_n^{\beta/\beta_i} \tilde{\epsilon}_n^{2 b_1},
\, i \in J^c\}$ 
\[
U_{j|k} \subset \left\{ \tilde{x}: \sum_{i=1}^{d_{J^c}} \left[(\tilde{x}_i-\mu^\star_{j|k,i})/(\sigma_n^{\beta/\beta_{d_J+i}} \tilde{\epsilon}_n^{2b_1})\right]^2\leq 1 \right\};\]
for $j=N+1,\ldots,K$,
$U_{j|k}$ is   contained within an ellipsoid 
with 
%center $\mu^\star_{j|k}$ and 
radii
$\{ \sigma_n^{\beta/\beta_i}, \, i \in J^c\}$,
and $1 \leq N < K \leq C_1 \sigma_n^{-d_{J^c}} \{\log (1/ \tilde{\epsilon}_n) \}^{d_{J^c}+d_{J^c}/\tau}$, 
where $C_1>0$ does not depend on $n$ and $y_J$.

Second, for each $k\in\mathcal{Y}_J$ there exist
$\alpha_{j|k}^\star$, $j = 1,\ldots,K$,  with $\alpha_{j|k}^\star=0$ for $j > N$,
and $\mu_{jk}^{x\star} \in U_{j|k}$ for $j=N+1,\ldots,K$
such that for a positive constant $C_{2}$ and
$\sigma^\star_{J^c}=\{\sigma_n^{\beta/\beta_i}$ for $i\in J^c\}$, 
\begin{equation}
\label{eq:f0upsbeta}
d_H\left(f_{0|J}(\cdot|k),f_{|J}^\star(\cdot|k)\right) \leq C_{2} \sigma_n^\beta,
\end{equation}
where $f_{|J}^\star$ is defined in \eqref{eq:f_J_def}.
Constant $C_{2}$ is the same for all $k \in \mathcal{Y}_J$ since all the bounds on $f_{0|J}$ assumed in Section \ref{sec:assumpt_dgp} are uniform over $k$. 

Note also that our smoothness definition is different from the one used by \cite{ShenTokdarGhosal2013}. 
In Lemmas \ref{lm:AnisoTaylor} and \ref{lm:AnisoTaylorRemainderBd} we show that our smoothness definition ($ f_{0|J} \in \mathcal{C}^{L,\beta_{d_J+1},\ldots,\beta_{d}})$ delivers an anisotropic Taylor expansion with bounds on remainder terms such that the argument on p. 637 of \cite{ShenTokdarGhosal2013} goes through.

Third, by Lemma \ref{lm:ProbBoundOutside}, which is an extension of a part of Proposition 1 in \cite{ShenTokdarGhosal2013},
 there exists a constant $B_0>0$ such that for all $y_J\in\mathcal{Y}_{J}$
\begin{align}
\label{eq:tail_prob_bound}
F_{0|J}\left(||\tilde{X}|| > a_{\sigma_n}| y_J \right)\leq B_0 \sigma_n^{4\beta+2\varepsilon}\underline{\sigma}_{n} ^8,
\end{align}
where 
\[
\underline{\sigma}_{n} = \min_{i\in J^c}\sigma_{n}^{\beta/\beta_i}. 
\]

For $m=N_JK$ we define $\theta^\star$ and $S_{\theta^\star}$ as:
\begin{align*}
\theta^\star  = \bigg\{ 
& \{ \mu_1^\star,\ldots,\mu_m^\star \} = 
\left\{ 
(k,\mu_{j|k}^{\star}),\:
j=1,\ldots,K, \; k\in\mathcal{Y}_J \right\},
\\
& \{\alpha_1^\star,\ldots,\alpha_m^\star\}  
%=\left\{\alpha_{jk}^\star \: j=1,\ldots,K, \; k\in\mathcal{Y}_J \right\}
= \left\{ \alpha_{jk}^\star=\alpha_{j|k}^{\star}\pi_{0J}(k), \:
j=1,\ldots,K, \; k\in\mathcal{Y}_J \right\},
\\
&\sigma_{J}^{\star2}=\{\sigma_i^{\star2}=1/[64 N_i^2 \beta \log (1/\sigma_n)], \: i \in J \}
\\
&\sigma_{J^c}^\star=\{\sigma_i^\star=\sigma_{n}^{\beta/\beta_i}, \: i \in J^c \},
\bigg \}
\end{align*}
\begin{align*}
S_{\theta^\star} =  
\bigg\{ 
& \{\mu_1,\ldots,\mu_m\} =  
\left\{ 
(\mu_{jk,J},\mu_{jk,J^c}), \:
j=1,\ldots,K, \, k \in \mathcal{Y}_J \right\},
\\
& 
\mu_{jk,J^c} \in U_{j|k}, \;\; \mu_{jk,i} \in  \left[ k_i - \frac{1}{4N_i}, k_i + \frac{1}{4N_i}\right] , \, i \in J,
\\
& \sigma_i^2 \in \left( 0, \sigma_{i}^{\star2} \right), \,  i \in J,
\\
&\sigma_i^2 \in \left(\sigma_{i}^{\star2}(1+\sigma_n^{2\beta})^\inv, \sigma_{i}^{\star2}\right), \,  i \in J^c,
\\
& (\alpha_1,\ldots,\alpha_m) = 
\{ \alpha_{jk}, \, j=1,\ldots,K, \, k\in\mathcal{Y}_J \} \in \Delta^{m - 1}, \\
& 
\sum_{r=1}^{m}
|\alpha_{r}-\alpha_{r}^\star| \leq  2 \sigma_n^{2\beta},\:\:\:\:\:
%\min_{k\in\mathcal{Y}_J} \sum_{j=1}^{K} \alpha_{jk} \geq \frac{\sigma_n^{d_{J^c}}}{N_J}
%\cdot 
%%C_{12}^\inv
%\bar{f}_0 \bar{\pi} (2\pi)^{d_{J^c}/2} 2
%,\\
%&
\min_{j\leq K, k\in\mathcal{Y}_J}   \alpha_{jk} \geq  \frac{\sigma_n^{2\beta + d_{J^c}}}{2m^2}
 \bigg\}.
\end{align*}
%for a positive constant $C_{12}$ defined later.

%Some extra notation on the conditional probabilities implied by the mixture of normals model for the chosen $\theta^\star$ and $m=N_JK$:
%\begin{align*}
%p_{|J}(y_I,x |y_J,\theta^\star,m)&=\int_{A_{y_I}} f_{|J}(\tilde{y}_I,x |y_J,\theta^\star,K) d\tilde{y}_I
%\\
%f_{J}(y_J,\tilde{x} |\theta^\star,m)&=\int_{A_{y_J}} f(\tilde{y}_J,\tilde{y}_I,x |\theta^\star,m) d\tilde{y}_J
%\end{align*}

The rest of the proof of the Kullback-Leibler thickness condition follows the general argument developed for mixture models in \cite{GhosalVandervaart:07} and \cite{ShenTokdarGhosal2013} among others.
First, we will show that for $m=N_JK$ and $\theta\in S_{\theta^\star}$, the Hellinger distance $d_H^2(p_0(\cdot,\cdot), p(\cdot,\cdot | \theta, m ))$ can be bounded by
$\sigma_n^{2\beta}$ up to a multiplicative constant.  Second, we construct bounds on the ratios $p(\cdot,\cdot | \theta, m )/p_0(\cdot,\cdot)$ and combine them with the bound on the Hellinger distance using 
Lemma \ref{lm:dH_KL}.  Finally, we will show that the prior puts  sufficient probability on  $m=N_JK$ and $S_{\theta^\star}$.

For $f_{|J}^\star$ defined in \eqref{eq:f_J_def}, let us define 
\[
p_{|J}^\star(y_I,x |y_J)=\int_{A_{y_I}} f_{|J}^\star(\tilde{y}_I,x |y_J) d\tilde{y}_I.
\]
For $m=N_JK$ and $\theta\in S_{\theta^\star}$, we can bound the Hellinger distance between the DGP and the model as follows,
\begin{align*}
&d_H^2(p_0(\cdot,\cdot), p(\cdot,\cdot | \theta, m ))  =d_H^2(p_{0|J}(\cdot|\cdot)\pi_0(\cdot), p(\cdot,\cdot | \theta, m )) \notag  \\
&\leq
d_H^2(p_{0|J}(\cdot|\cdot)\pi_{0J}(\cdot), p_{|J}^\star(\cdot|\cdot)\pi_{0J}(\cdot)) +
d_H^2(p_{|J}^\star(\cdot|\cdot)\pi_{0J}(\cdot), p(\cdot,\cdot | \theta, m )).
\end{align*}
It follows from \eqref{eq:f0upsbeta} and 
Lemma \ref{lm:p_f_distance_bds} linking distances between probability mass functions and corresponding latent variable densities that the first term on the right hand side of this inequality is bounded by 
$(C_2)^2\sigma_n^{2\beta}$.
Combining this result with the bound on $d_H^2(p_{|J}^\star(\cdot|\cdot)\pi_{0J}(\cdot), p(\cdot,\cdot | \theta, m ))$ from Lemma \ref{lm:H_dist_model_pi0_model} we obtain
\begin{align}
\label{eq:H2_bound_dgp_model}
d_H^2(p_0(\cdot,\cdot), p(\cdot,\cdot | \theta, m )) \lesssim \sigma_n^{2\beta}.
\end{align}

Next, for $\theta \in S_{\theta^\star}$ and $m=N_JK$, let us consider lower bounds on the ratio $p(y_J,y_I,{x} | \theta, m)/p_{0}(y_J,y_I,x)$.
In Lemma \ref{lm:Bound_Model_truef_ratio} in the Appendix we show that 
lower bounds on the ratio 
$f_J(y_J,\tilde{x} | \theta, m)/f_{0|J}(\tilde{x} | y_J)\pi_{0}(y_J)$
imply the following bounds for all sufficiently large $n$:
for any ${x} \in {\mathcal{{X}}}$  with $\norm{{x}} \leq a_{\sigma_n}$,
\begin{align}
\label{eq:define_lambda_n_MAIN}
\frac{p(y_J,y_I,{x} | \theta, m)}{p_{0}( y_J,y_I,x)} 
\geq
C_3  \frac{\sigma_n^{2\beta }}{2m^2}
\equiv \lambda_n,
\end{align}
for some constant $C_3>0$; and for any ${x} \in {\mathcal{{X}}}$  with $\norm{{x}} > a_{\sigma_n}$,
\begin{align}
\label{eq:boundratio_forlarge_x_MAIN}
\frac{p(y_J,y_I,{x} | \theta, m)}{p_{0}( y_J,y_I,x)} 
 \geq 
\exp\left\{ - \frac{8||{x}||^2}{\underline{\sigma}_{n}^2} - C_4 \log n \right\},
\end{align}
for some constant $C_4>0$.
Consider all sufficiently large $n$ such that 
$\lambda_n < e^{-1}$
and \eqref{eq:define_lambda_n_MAIN}
and \eqref{eq:boundratio_forlarge_x_MAIN} hold.
Then, for any $\theta \in S_{\theta^\star}$,
\begin{align}
\label{eq:exp_log_ratio2}
&
\sum_{y\in\mathcal{Y}}
 \int_{{\mathcal{X}}} 
\bigg( \log \frac{p_0(y_J,y_I,x)}{p(y_J,y_I,{x}| \theta, m)}\bigg)^2 
\indic \left \{ \frac{p(y_J,y_I,{x}| \theta, m)}{p_0(y_J,y_I,x)} < \lambda_n \right\}
p_0(y_J,y_I,x) d{x} \notag \\
&=
\sum_{y\in\mathcal{Y}}
\int_{\tilde{\mathcal{X}}} 
\bigg( \log \frac{p_0(y_J,y_I,x)}{p(y_J,y_I,{x}| \theta, m)}\bigg)^2 
\indic \left \{ \frac{p(y_J,y_I,{x}| \theta, m)}{p_0(y_J,y_I,x)} < \lambda_n
%,\tilde{y}_I\in A_{y_I}
 \right\}
\indic \left \{\tilde{y}_I\in A_{y_I} \right\}
f_{0J}(y_J,\tilde{x}) d\tilde{x} \notag \\
& =
\sum_{y\in\mathcal{Y}}
 \int_{\tilde{\mathcal{X}}}
\bigg( \log \frac{p_0(y_J,y_I,x)}{p(y_J,y_I,{x}| \theta, m)}\bigg)^2 
\indic \left \{ \frac{p(y_J,y_I,{x}| \theta, m)}{p_0(y_J,y_I,x)} < \lambda_n , ||{x}||> a_{\sigma_n}
,\tilde{y}_I\in A_{y_I}
\right\}
%\indic \left \{\tilde{y}_I\in A_{y_I} \right\}
f_{0J}(y_J,\tilde{x}) d\tilde{x} \notag\\
& \leq
\sum_{y\in\mathcal{Y}}
 \int_{\left\{\tilde{x}:||{x}||>a_{\sigma_n}\right\}}
\bigg( \log \frac{p_0(y_J,y_I,x)}{p(y_J,y_I,{x}| \theta, m)}\bigg)^2 
\indic \left \{\tilde{y}_I\in A_{y_I} \right\}
f_{0J}(y_J,\tilde{x}) d\tilde{x} 
 \notag
\\
& \leq
\sum_{y\in\mathcal{Y}}
 \int_{\left\{\tilde{x}:||{x}||>a_{\sigma_n}\right\}} \left [
\frac{128}{\underline{\sigma}_n^4} ||{x}||^4 + 2(C_4 \log n)^2\right]f_{0|J}(\tilde{x}|y_J)
\indic \left \{\tilde{y}_I\in A_{y_I} \right\}
 d\tilde{x}
\pi_{0J}(y_J)\notag \\
&
\leq 
\sum_{y_J\in\mathcal{Y}_J}
 \int_{\left\{\tilde{x}:||\tilde{x}||>a_{\sigma_n}\right\}} 
\left [
\frac{128}{\underline{\sigma}_n^4} ||\tilde{x}||^4 + 2(C_4 \log n)^2\right]
 f_{0|J}(\tilde{x}|y_J) d\tilde{x}
\pi_{0J}(y_J)
\notag \\
&\leq
\frac{128}{\underline{\sigma}_n^4}
 \sum_{y_J\in\mathcal{Y}_J} E_{0|y_J}\left(\norm{\tilde{X}}^8\right)^{1/2} \left(F_{0|y_J}\left(\norm{\tilde{X}} > a_{\sigma_n}\right)\right)^{1/2} \pi_{0J}(y_J) 
+ 2(C_4 \log n)^2 B_0 \sigma_n^{4\beta+2\varepsilon}\underline{\sigma}_{n} ^8
\notag
\\
& \leq C_5 \sigma_n^{2\beta + \varepsilon}
 \end{align}
  for some constant $C_{5}>0$ and 
	all sufficiently large $n$,
	where the last inequality holds by the tail condition in \eqref{eq:tail_cond2}, \eqref{eq:tail_prob_bound}, and $(\log n)^2\sigma_n^{2\beta+\varepsilon}\underline{\sigma}_{n}^8 \rightarrow 0$. 
  
 Furthermore, as $\lambda_n< e^\inv$, 
 \begin{align*}
& \log \frac{p_0(y_J,y_I,x)}{p(y_J,y_I,{x}| \theta, m)}
\indic \left \{ \frac{p(y_J,y_I,{x}| \theta, m)}{p_0(y_J,y_I,x)} < \lambda_n \right\}
\\
&\leq
 \bigg( \log \frac{p_0(y_J,y_I,x)}{p(y_J,y_I,{x}| \theta, m)}\bigg)^2 
\indic \left \{ \frac{p(y_J,y_I,{x}| \theta, m)}{p_0(y_J,y_I,x)} < \lambda_n \right\} 
\end{align*}
 and, therefore,
 \begin{align}
\label{eq:exp_log_ratio}
\sum_{y\in\mathcal{Y}}
 \int_{{\mathcal{X}}} 
 \log \frac{p_0(y_J,y_I,x)}{p(y_J,y_I,{x}| \theta, m)}
\indic \left \{ \frac{p(y_J,y_I,{x}| \theta, m)}{p_0(y_J,y_I,x)} < \lambda_n \right\}
p_0(y_J,y_I,x) d{x}
\leq C_{5} \sigma_n^{2\beta + \varepsilon}.
 \end{align}

Inequalities 
\eqref{eq:H2_bound_dgp_model},  \eqref{eq:exp_log_ratio2}, and \eqref{eq:exp_log_ratio}
combined with Lemma \ref{lm:dH_KL} imply
\[
E_0 \left(  \log \frac{p_0(y_J,y_I,x)}{p(y_J,y_I,{x}| \theta, m)} \right)\leq A\tilde{\epsilon}_n^2, \;
E_0 \left( \left[   \log \frac{p_0(y_J,y_I,x)}{p(y_J,y_I,{x}| \theta, m)} \right]^2 \right) \leq A\tilde{\epsilon}_n^2
\]
for any $\theta \in S_{\theta^\star}$, $m=N_JK$, and some positive constant $A$ (details are provided in Lemma \ref{lm:gKL_inequality} in the Appendix).

By Lemma \ref{lm:prior_bound} in the Appendix
for all sufficiently large $n$,  $s = 1 + 1/\beta + 1/\tau$, and some $C_6>0$,
\begin{eqnarray*}
\Pi( \mathcal{K}(p_0, \tilde{\epsilon}_n))  \geq \Pi(m=N_JK, \theta \in S_{\theta^\star} )  \geq 
\exp \left[-C_6 N_J\tilde{\epsilon}_n^{-d_{J^c}/\beta} \{\log (n)\}^{d_{J^c}s + \max\{\tau_1,1,\tau_2/\tau\} }\right] .
\end{eqnarray*}
The last expression of the above display is bounded below by $\exp\{-C n \tilde{\epsilon}_n^2 \}$ for any $C>0$,
$\tilde{\epsilon}_n = \left[ \frac{N_J}{n}\right]^{\beta/(2\beta + d_{J^c})} (\log n)^{t_J}$, any $t_J > (d_{J^c}s + \max\{\tau_1,1,\tau_2/\tau\}) / (2 + d_{J^c}/\beta)$, and all sufficiently large $n$.  Since the inequality in the definition of $t_J$ is strict, the claim of the theorem follows.

When $J = \emptyset$ and $N_J=1$, the preceding argument delivers the claim of the theorem if 
we add an artificial discrete coordinate with only one possible value to the vector of observables.

\subsubsection{Proof of Theorem \ref{th:prior_thickness} for $J^c= \emptyset$}

%For the case of $J^c=\emptyset$ the proof can be simplified as follows. 
In this case, the proof from the previous subsection can be simplified as follows. 
For $m=N_J$ and for any $\beta>0$ we define $\theta^\star$ and $S_{\theta^\star}$ as
\begin{align*}
\theta^\star= \bigg\{
&\{\mu_1^\star,\ldots,\mu_m^\star\} = \left\{ 
k, \; k\in\mathcal{Y}_J \right\},
\\
&\{\alpha_1^\star,\ldots,\alpha_m^\star\} = \left\{ 
\alpha_{k}^\star, \; k\in\mathcal{Y}_J \right\}
=  \left\{ 
\pi_{0}(k)
\right\}_{k\in\mathcal{Y}_J},
\\
&\sigma^{\star2}=\{\sigma_i^{\star2} = \frac{1}{64 N_i^2 \beta \log (1/\sigma_n)},\; i \in J\}
\bigg\},
\end{align*}
\begin{align*}
S_{\theta^\star} =  
&\bigg\{ 
\{\mu_1,\ldots,\mu_m\} = \left\{ 
\mu_{k}, \; k\in\mathcal{Y}_J
\right\}, \;
  \mu_{k,i} \in  \left[ k_i - \frac{1}{4N_i}, k_i + \frac{1}{4N_i}\right], \; i=1,\ldots,d_J,\\
&\sigma= \{\sigma_i \in (0, \sigma_i^\star), \;  i \in J\},\\
& \{\alpha_{j}, \; j=1,\ldots,m\} = \{\alpha_{k}, \; k \in \mathcal{Y}_J\} \in \Delta^{m - 1},\\
& \sum_{k\in\mathcal{Y}_J} 
|\alpha_{k}-\alpha_{k}^\star| \leq  2 \sigma_n^{2\beta}, \;\;\;\;\;
\min_{k\in\mathcal{Y}_J}   \alpha_{k} \geq  \frac{\sigma_n^{2\beta }}{2m^2}
 \bigg\}.
\end{align*}

For $m=N_J$ and $\theta\in S_{\theta^\star}$,
a simplification of the proof of Lemma \ref{lm:H_dist_model_pi0_model} delivers
\begin{align*}
d_H^2(p_0(\cdot), p(\cdot | \theta, m ))  
\leq 
2 \max_{k\in\mathcal{Y}_J} \int_{A^c_{k}}\phi(\tilde{y}_J;\mu_{k},\sigma) d\tilde{y}_J
+
\sum_{k\in\mathcal{Y}_J} 
\left|
\alpha_{k}^\star
-
\alpha_{k} 
\right| \notag \lesssim \sigma_n^{2\beta}. \notag
\end{align*}

A simplification of derivations in Lemma \ref{lm:Bound_Model_truef_ratio} show that for all $y_J\in\mathcal{Y}_J$
\begin{align*}
\frac{p(y_J | \theta, m)}{p_{0}( y_J)} 
\geq
\frac{1}{2}  \frac{\sigma_n^{2\beta }}{2m^2}
\equiv \lambda_n. \notag
\end{align*}

Then,
for any $\theta \in S_{\theta^\star}$
\begin{align}
\sum_{y_J\in\mathcal{Y}_J}
\bigg( \log \frac{p_0(y_J)}{p(y_J| \theta, m)}\bigg)^2 
\indic \left \{ \frac{p(y_J| \theta, m)}{p_0(y_J)} < \lambda_n \right\}
p_0(y_J) &=0 \notag \\
\sum_{y_J\in\mathcal{Y}_J}
\bigg( \log \frac{p_0(y_J)}{p(y_J| \theta, m)}\bigg)
\indic \left \{ \frac{p(y_J| \theta, m)}{p_0(y_J)} < \lambda_n \right\}
p_0(y_J) &=0 \notag \\
 \end{align}
as $\frac{p(y_J| \theta, m)}{p_0(y_J)} \geq \lambda_n$ for all $y_J\in\mathcal{Y}_J$.
As $\lambda_n\rightarrow 0$, by Lemma \ref{lm:dH_KL} for $\lambda_n<\lambda_0$, 
  both
$E_0 (  \log \frac{p_0(y_J)}{p(y_J| \theta, m)} )$ and $E_0 ( [   \log \frac{p_0(y_J)}{p(y_J| \theta, m)} ]^2)$
 are bounded by  $C_{7} \log (1/\lambda_n)^2\sigma_n^{2\beta} \leq A\tilde{\epsilon}_n^2$ for some constant $A$.
By the simplification of Lemma \ref{lm:prior_bound} for this particular case
for all sufficiently large $n$ and some $C_{8}>0$,
\begin{eqnarray*}
\Pi( \mathcal{K}(p_0, \tilde{\epsilon}_n))  \geq \Pi(m=N_J, \theta \in S_{\theta^\star} )  \geq 
\exp \left[-C_{8}N_J \{\log (n)\}^{ \max\{\tau_1,1\} }\right] .
\end{eqnarray*}
The last expression of the above display is bounded below by $\exp\{-C n \tilde{\epsilon}_n^2 \}$ for any $C>0$,
$\tilde{\epsilon}_n = \left[ \frac{N_J}{n}\right]^{1/2} (\log n)^{t_J}$, any $t_J > \max\{\tau_1,1\} /2 $, and all sufficiently large $n$.  Since the inequality in the definition of $t_J$ is strict, the claim of the theorem follows.

%\end{proof}

\section{Future Work}
\label{sec:conclusion}

It seems feasible to extend the results of this paper to conditional density estimation by covariate dependent mixtures 
along the lines of \cite{norets_pati_2017}.  We leave this to future work.

\bibliographystyle{../../../latex_common/ecta}
\bibliography{../../../latex_common/allreferences}
%\bibliographystyle{ecta}
%\bibliography{allreferences}

\section*{Appendix}

\subsection{Proofs and Auxiliary Results for Lower Bounds}
\label{sec:proofs_extras}

\begin{lemma} 
\label{lm:q_tvd}
For $q_j$, $q_l$, $i\neq l$ defined in  \eqref{eq:q_j}, the total variation distance is bounded below by $\mbox{const} \cdot \Gamma_n$.
\end{lemma}
\begin{proof} %(Lemma \ref{lm:q_tvd})
Let us establish several facts about $g_r$ in the definition of $q_j$.
For any $(\tilde{y},x) \in [0,1]^d$, there exists $r(\tilde{y},x)$ such that 
\begin{equation}
\label{eq:ryx}
g_r(\tilde{y},x)=0, \, \forall r\neq r(\tilde{y},x).
\end{equation}
For $(\tilde{y},x) \in B_r$, $r(\tilde{y},x)=r$ and for $(\tilde{y},x) \notin \cup_{r=1}^{\bar{m}} B_r$,
$r(\tilde{y},x)$ can have an arbitrary value.
Thus,
\begin{align*}
d_{TV}(q_j,q_l)&=\sum_{y} \int   \left| 
 \int_{A_y} \left [\sum_{r=1}^{\bar{m}} (w^j_r-w^l_r) g_r(\tilde{y},x) \right ] d\tilde{y}
\right|dx \\
&= 
\sum_{y} \int   \left| \int_{A_y}
 (w^j_{r(\tilde{y},x)}-w^l_{r(\tilde{y},x)})  g_{r(\tilde{y},x)}(\tilde{y},x)  d\tilde{y}
\right|dx.
\end{align*}
From $h_i = (2/N_i) \cdot R_i$ for $i \{1,\ldots,d_y\}$, where $R_i$ is a positive integer, %(Lemma \ref{lm:b_bast_h_N})
and the definitions of $g$, $g_r$, and $A_y$, it follows that  
for fixed $y \in \mathcal{Y}$ and $x \in [0,1]^{d_x}$,
$(w^j_{r(\tilde{y},x)}-w^l_{r(\tilde{y},x)})  g_{r(\tilde{y},x)}(\tilde{y},x)$ does not change the sign
as $\tilde{y}$ changes within $A_y$ 
%(if $r(\tilde{y},x)$ changes then $g_{r(\tilde{y},x)}(\tilde{y},x)=0$ $\forall \tilde{y} \in A_y$).
($r(\tilde{y},x)$ is the same $\forall \tilde{y} \in A_y$ by the choice of $c_i^r$ and $h_i$).
Therefore,
\begin{align}
d_{TV}(q_j,q_l)& = 
\int    \int \left|
 (w^j_{r(\tilde{y},x)}-w^l_{r(\tilde{y},x)})  g_{r(\tilde{y},x)}(\tilde{y},x) \right| d\tilde{y}
dx \nonumber \\
& =
\sum_{r=1}^{\bar{m}} \int_{B_r} \left|
 (w^j_{r(z)}-w^l_{r(z)})  g_{r(z)}(z) \right| dz \nonumber
\\
\label{eq:dtv_int}
& = 
\sum_{r=1}^{\bar{m}} |w^j_{r}-w^l_{r}| \int_{B_r} \left|
   g_{r}(z) \right| dz.
\end{align}
Finally,
\begin{align*}
d_{TV}(q_j,q_l)&=
\sum_{r=1}^{\bar{m}} 1\{w^j_{r}\neq w^l_{r}\} \cdot 
\Gamma_n \cdot \prod_{i=1}^d h_i \cdot  \left[ \int_{-1/2}^{1/2} |g(u)| d u \right]^d
& \mbox{(change of variables in \eqref{eq:dtv_int})}
\\
&\geq 
\Gamma_n \cdot \prod_{i=1}^d m_i h_i \cdot  \left[ \int_{-1/2}^{1/2} |g(u)| d u \right]^d/8
& \mbox{(by Lemma \ref{lm:VarshamovGilbert})}
\\
&\geq 
\Gamma_n \cdot \left[ \int_{-1/2}^{1/2} |g(u)| d u /2\right]^d/8
& \mbox{(since $m_i h_i > 1/2$).}
\end{align*}

\end{proof}

\begin{lemma} 
\label{lm:q_kl}
For $\Gamma_n \rightarrow 0$ and $\bar{m}\geq 8$ and a sufficiently small $c_0$ in the definition of $g$, condition
\eqref{eq:avgKLcond} in Lemma \eqref{lm:AbstrLBTsyb} holds for all sufficiently large $n$.

%Let $f_1(z)=1_{[0,1]^d}$ and $0<  \underline{c} \leq f_2(z) \leq \overline{c} < \infty $ be densities on $[0,1]$.
%Then, $\int f_1 \log f_1/f_2 \leq \int (f_1-f_2)^2$
\end{lemma}

\begin{proof} %(Lemma \ref{lm:q_kl})
By Lemma \ref{lm:VarshamovGilbert}, it suffices to show that
\begin{equation}
\label{eq:suff_bdmdiv64}
d_{KL}(Q_j^n,Q_0^n)=n \cdot d_{KL}( q_j,q_0) < ( \bar{m}\log 2)/64.
\end{equation}

First, note that for a density $f\geq c > 0$ on $[0,1]^d$, $d_{KL}(f,1_{[0,1]^d})$ is bounded above by
\begin{equation}
\label{eq:flof_bd}
d_{KL}(f,1_{[0,1]^d}) + d_{KL}(1_{[0,1]^d},f)
= \int_{[0,1]} (f-1) \log f \leq \frac{\int_{[0,1]} (f-1)^2}{c}.
\end{equation}
Next, note that for any $z \in [0,1]^d$, the density in the definition of $q_j$
\begin{equation}
\label{eq:qj_lb}
1_{[0,1]^d}(z) + \sum_{r=1}^{\bar{m}} w^j_r g_r(z)	\geq 1 - \Gamma_n \left[\max_{u \in [-1/2,1/2] } g(u) \right]^d \geq 1/2
\end{equation}
for all sufficiently large $n$.
Thus,
\begin{align}
& d_{KL}(q_j,q_0)  \leq d_{KL}\left(1_{[0,1]^d} + \sum_{r=1}^{\bar{m}} w^j_r g_r, 1_{[0,1]^d}\right) & \mbox{ (by \eqref{eq:dKLp_dKLf})}  \nonumber \\
& \leq 2 \int \left [ \sum_{r=1}^{\bar{m}} w^j_r g_r(z) \right ]^2 dz & \mbox{ (by \eqref{eq:flof_bd} and \eqref{eq:qj_lb})} \nonumber \\
& = 2 \int \sum_{r=1}^{\bar{m}} w^j_r (g_r(z))^2 dz & \mbox{ (since  $g_r(z)g_l(z)=0, \forall r \neq l$)} \nonumber \\
& \leq 2 \bar{m} \int (g_1(z))^2 dz = 2 \Gamma_n^2 \prod_i 
%(m_i/h_i)
(m_ih_i)
\left[\int_{-1/2}^{1/2} g(u)^2 du\right]^d & \nonumber \\
& \leq 2 \Gamma_n^2 \left[\int_{-1/2}^{1/2} g(u)^2 du\right]^d \leq 2 \Gamma_n^2 c_0^{2d}.
\label{eq:dkl_geq_G2}
\end{align}
Finally, 
\begin{align*}
\bar{m} &=\prod_{i=1}^d m_i \geq 2^{-d} \prod_{i=1}^d h_i^{-1} & \mbox{ (by definitions of $\bar{m}$ and $m_i$)}
\\
& 
%=
%2^{-d} \prod_{i \in J_\ast} (N_i/8) \cdot \prod_{i \in J_\ast^c} \left(\Gamma_n^{\beta_i^{-1}}/8\right)  &   \mbox{ (by definition of $h_i$)} \\
%& =
%2^{-4d} \cdot N_{J_\ast} \cdot \Gamma_n^{-\beta_{J_\ast^c}^{-1}} = 2^{-4d} n \Gamma_n^2 & \\
%& \geq 2^{-4d} n \cdot d_{KL}(q_j,q_0) / (2 c_0^{2d}). & \mbox{ (by \eqref{eq:dkl_geq_G2})}
=
2^{-d} \prod_{i \in J_\ast} (N_i/2) \cdot \prod_{i \in J_\ast^c, i\leq d_y} \left(\Gamma_n^{-\beta_i^{-1}}/\varrho_i\right)\cdot \prod_{i \in J_\ast^c, i> d_y} \left(\Gamma_n^{-\beta_i^{-1}}\right) &   \mbox{ (by definition of $h_i$)} \\
&
\geq
2^{-d} \prod_{i \in J_\ast} (N_i/2) \cdot \prod_{i \in J_\ast^c, i\leq d_y} \left(\Gamma_n^{-\beta_i^{-1}}/2\right) \cdot \prod_{i \in J_\ast^c, i> d_y} \left(\Gamma_n^{-\beta_i^{-1}}\right) &   \mbox{ (by restrictions on $\varrho_i$)} \\
& 
=
2^{-d-d_y} \cdot N_{J_\ast} \cdot \Gamma_n^{-\beta_{J_\ast^c}^{-1}} = 2^{-d-d_y} n \Gamma_n^2 & \\
& 
\geq 2^{-d-d_y} n \cdot d_{KL}(q_j,q_0) / (2 c_0^{2d}) & \mbox{ (by \eqref{eq:dkl_geq_G2}).}
\end{align*}
The last inequality implies \eqref{eq:suff_bdmdiv64} if
\[
%c_0 \leq 2^{-2-3.5/d}.
%c_0 \leq 2^{-1-2.5/d}.
c_0 \leq [2^{-(d+d_y+7)} \log2]^{1/(2d)}.
\]
\end{proof}

\begin{lemma} 
\label{lm:q_smooth}
For $j \in \{0,\ldots,M\}$, $q_j \in \mathcal{C}^{\beta_1^\ast,\ldots,\beta_d^\ast,L}$  with $L=1$ for
any sufficiently small constant $c_0$ in the definition of $g$.
%Also, $\beta_i^\ast \geq \beta_i$, and, thus, $q_j \in \mathcal{C}^{\beta_1,\ldots,\beta_d,L}$.
\end{lemma}
\begin{proof} %(Lemma \ref{lm:q_smooth})
For $j=0$, the result is trivial.  For $j\neq 0$, consider
$k=(k_1,\ldots,k_d)$ and $z, \Delta z \in \mathbb{R}^d$ such that
for some $ i \in \{1,\ldots,d\}$, $\Delta z_i \neq 0$, for any $l\neq i$,  $\Delta z_l = 0$,
$\sum_{l=1}^d k_l/\beta_l^\ast<1$,
and $\sum_{l=1}^d k_l/\beta_l^\ast + 1/\beta_i^\ast\geq 1$ so that
\begin{equation}
\label{eq:b1_sb_leq1}
0 \leq \beta_i^\ast (1-\sum_{l=1}^dk_l/ \beta_l^\ast)	\leq 1.
\end{equation}
For $r(\cdot)$ defined in \eqref{eq:ryx},
\begin{align}
\label{eq:Dk_q}
D^{k} q_j(z) &= 1\{k=(0,\ldots,0)\}+w_{r(z)} \Gamma_n \prod_{l=1}^d g^{(k_l)}((z_l-c^{r(z)}_l)/h_l)/h_l^{k_i}
\nonumber \\
&= 1\{k=(0,\ldots,0)\}+
B_i \cdot w_{r(z)} h_i^{\beta_i^\ast(1-\sum_{l=1}^d k_l/\beta_l^\ast)} \prod_{l=1}^d g^{(k_l)}((z_l-c^{r(z)}_l)/h_l),
\end{align}
where $B_i \in \{1,1/2,\varrho_i^{-\beta_i^\ast} \} \subset (0,1]$.
In what follows we consider $k\neq (0,\ldots,0)$ to simplify the notation;
when $k=(0,\ldots,0)$ the argument below goes through as the indicator function 
$1\{k=(0,\ldots,0)\}$ is canceled out in the differences of derivatives.
From \cite{Tsybakov:08}, (2.33)-(2.34), for any sufficiently small $c_0$ and $s \leq \max_l{\beta_l^\ast}+1$, 
\begin{equation}
\label{eq:gk_leq18}
\max_z|g^{(s)}(z)|\leq 1/8.
\end{equation}
This imply that
\begin{equation}
\label{eq:gk_lipschitz}
|g^{(k_i)}((z_i+\Delta z_i-c^{r}_i)/h_i) - 
g^{(k_i)}((z_i-c^{r}_i)/h_i)| \leq
| \Delta z_i|/(8h_i).
\end{equation}

First, let us consider the case when $r(z)=r(z+\Delta z)$ and $|\Delta z_i| \leq h_i$.
From \eqref{eq:Dk_q}, \eqref{eq:gk_leq18}, and \eqref{eq:gk_lipschitz},
%\begin{align*}
\begin{align}
\label{eq:diff_for_small_Deltaz}
|D^{k} q_j(z+\Delta z) - D^{k} q_j(z)| & \leq 
h_i^{\beta_i^\ast(1-\sum_{l=1}^d k_l/\beta_l^\ast)} 8^{-d} |\Delta z_i/h_i| \notag \\
%& = 8^{-d} [\Delta z_i]^{\beta_i^\ast(1-\sum_{l=1}^d k_l/\beta_l^\ast)} 
%\left [\frac{\Delta z_i}{h_i}\right]^{1-\beta_i^\ast(1-\sum_{l=1}^d k_l/\beta_l^\ast)} 
& = 8^{-d} |\Delta z_i|^{\beta_i^\ast(1-\sum_{l=1}^d k_l/\beta_l^\ast)} 
\left |\frac{\Delta z_i}{h_i}\right|^{1-\beta_i^\ast(1-\sum_{l=1}^d k_l/\beta_l^\ast)} 
\notag
\\
& \leq |\Delta z_i|^{\beta_i^\ast(1-\sum_{l=1}^d k_l/\beta_l^\ast)},
\end{align}
%\end{align*}
where the last inequality follows from $\Delta z_i \leq h_i$ and \eqref{eq:b1_sb_leq1}.

Second, consider the case when $r(z)=r(z+\Delta z)$ and $|\Delta z_i |> h_i$.  Similarly to the previous case but without using \eqref{eq:gk_lipschitz},
\begin{align*}
|D^{k} q_j(z+\Delta z) - D^{k} q_j(z)|  \leq 
2 \cdot 8^{-d} h_i^{\beta_i^\ast(1-\sum_{l=1}^d k_l/\beta_l^\ast)} 
 \leq |\Delta z_i|^{\beta_i^\ast(1-\sum_{l=1}^d k_l/\beta_l^\ast)}.
\end{align*}

%Third, when $r(z)\neq r(z+\Delta z)$ and $\Delta z_i \leq h_i/4$,
%\[
%|D^{k} q_j(z+\Delta z) - D^{k} q_j(z)| = D^{k} q_j(z+\Delta z) = D^{k} q_j(z) = 0
%\]
%by construction of $q_j$, $g_r$, and $g$.
Third, consider the case when $r(z)\neq r(z+\Delta z)$ and $|\Delta z_i| \leq h_i/2$. If $w_{r(z)}= w_{r(z+\Delta z)}=0$ or $z,z+\Delta{z} \notin \cup_{r=1}^{\bar{m}} B_r$ 
\[
|D^{k} q_j(z+\Delta z) - D^{k} q_j(z)| = D^{k} q_j(z+\Delta z) = D^{k} q_j(z) = 0.
\]
If $w_{r(z)}\neq w_{r(z+\Delta z)}$ or if one of $z$ and $z+\Delta z$ is not in $\cup_{r=1}^{\bar{m}} B_r$, then without a loss of generality suppose that $w_{r(z)}=1$ or that $z+\Delta z \notin \cup_{r=1}^{\bar{m}} B_r$. 
Let $|\Delta z_i^\star|\in[0,|\Delta z_i|]$ and $\Delta z^\star = (0,\ldots,0,\Delta z_i^\star,0,\ldots,0)$ be such that $z+\Delta z^\star$ is a boundary point of $B_{r(z)}$. 
Then, $D^{k} q_j(z+\Delta z^\star)=0$ and (\ref{eq:diff_for_small_Deltaz}) imply
\begin{align*}
|D^{k} q_j(z+\Delta z) - D^{k} q_j(z)| &= |D^{k} q_j(z)| = |D^{k} q_j(z+\Delta z^\star) - D^{k} q_j(z)| \\
 & \leq |\Delta z_i^\star|^{\beta_i^\ast(1-\sum_{l=1}^d k_l/\beta_l^\ast)}
\leq |\Delta z_i|^{\beta_i^\ast(1-\sum_{l=1}^d k_l/\beta_l^\ast)}.
\end{align*}
If  $w_{r(z)}= w_{r(z+\Delta z)}=1$ and $z,z+\Delta{z} \in \cup_{r=1}^{\bar{m}} B_r$ then by construction of $q_j$ and $g$
\begin{align*}
|D^{k} q_j(z+\Delta z) - D^{k} q_j(z)| &= 
|D^{k} q_j(z+\Delta z+0.5h_i) - D^{k} q_j(z+0.5h_i)| 
\leq |\Delta z_i|^{\beta_i^\ast(1-\sum_{l=1}^d k_l/\beta_l^\ast)},
\end{align*}
where the last inequality follows from (\ref{eq:diff_for_small_Deltaz}).

Finally, when $r(z)\neq r(z+\Delta z)$ and 
$\Delta z_i > h_i/2$,
\begin{align*}
& |D^{k} q_j(z+\Delta z) - D^{k} q_j(z)|  \leq |D^{k} q_j(z+\Delta z)| + |D^{k} q_j(z)| \\
& \leq 2 \cdot 8^{-d} h_i^{\beta_i^\ast(1-\sum_{l=1}^d k_l/\beta_l^\ast)} \\
& \leq |\Delta z_i|^{\beta_i^\ast(1-\sum_{l=1}^d k_l/\beta_l^\ast)}.
\end{align*}

Now, let us consider a general $\Delta z$ such that for $\Delta z_i \neq 0$, 
$\sum_{l=1}^d k_l/\beta_l^\ast + 1/\beta_i^\ast \geq 1$.
\begin{align*}
& |D^{k} q_j(z+\Delta z) - D^{k} q_j(z)| \\ \leq &
\sum_{i=1}^d 
|D^{k} q_j(z_1, \ldots, z_{i-1}, z_i+\Delta z_i,\ldots,z_d+\Delta z_d) - D^{k} q_j(z_1, \ldots, z_{i}, z_{i+1}+\Delta z_{i+1},\ldots,z_d+\Delta z_d)|.
%|D^{k} q_j(z+\Delta z) - D^{k} q_j(z_1, z_2+\Delta z_2,\ldots,z_d+\Delta z_d)| \\
%&+
%|D^{k} q_j(z_1, z_2+\Delta z_2,\ldots,z_d+\Delta z_d) - D^{k} q_j(z_1, z_2,z_3+\Delta 3_d,\ldots,z_d+\Delta z_d)| + \cdots \\
%&+
%|D^{k} q_j(z_1, z_2,\ldots,z_{d-1}, z_d+\Delta z_d) - D^{k} q_j(z_1, \ldots,z_d)|
\end{align*}
The preceding argument applies to every term in this sum and, thus, $q_j \in \mathcal{C}^{\beta_1^\ast,\ldots,\beta_d^\ast, 1}$.

\end{proof}

\begin{lemma} 
\label{lm:p_f_distance_bds}
Let $f_i:\tilde{\mathcal{Y}}\times \mathcal{X} \rightarrow \mathbb{R}$, $i \in \{1,2\}$, be densities with respect to a product measure $\lambda\times\mu$ on $\tilde{\mathcal{Y}}\times \mathcal{X} \subset \mathbb{R}^d$.
For a finite set $\mathcal{Y}$, let $\{A_y, \, y \in  \mathcal{Y}\}$ be a partition of $\tilde{\mathcal{Y}}$ and let $p_i(y,x)= \int_{A_y} f_i(\tilde{y}, x) d\lambda(\tilde{y})$.
%be the corresponding probability mass-densities on $\mathcal{Y}\times \mathcal{X}$.  
Then,
\begin{align}
\label{eq:dTVp_dTVf}
d_{TV}(p_1,p_2) &\leq d_{TV}(f_1,f_2)
\\
\label{eq:dHp_dHf}
d_H(p_1,p_2) &\leq d_H(f_1,f_2)
\\
\label{eq:dKLp_dKLf}
d_{KL}(p_1,p_2) &\leq d_{KL}(f_1,f_2).
\end{align}
Also, if for given $(y,x)$, $f_2(\tilde{y}, x)>0$ for any $\tilde{y} \in A_y$, then
\begin{equation}
\label{eq:p_f_ratiobds}
\inf_{\tilde{y} \in A_y} \frac{f_1(\tilde{y},x)}{f_2(\tilde{y},x)} \leq \frac{p_1(y,x)}{p_2(y,x)} \leq \sup_{\tilde{y} \in A_y} \frac{f_1(\tilde{y},x)}{f_2(\tilde{y},x)}
.
\end{equation}
\end{lemma}
\begin{proof} %(Lemma \ref{lm:p_f_distance_bds})
Trivially,
\begin{align*}
& d_{TV}(p_1,p_2) 
=
\sum_y \int 
\left |\int_{A_y} (f_1(\tilde{y},x) - f_2(\tilde{y},x)) d \tilde{y} \right |
d \mu(x) \\
& \leq 
\sum_y \int 
\int_{A_y} | f_1(\tilde{y},x) - f_2(\tilde{y},x)| d \lambda(\tilde{y})
d \mu(x)
=d_{TV}(f_1,f_2).
\end{align*}
By Holder inequality,
\begin{align*}
& d_{H}(p_1,p_2) 
= 2\left(1- \sum_y \int \sqrt{ \int  1_{A_y}(\tilde{y}_1)f_1(\tilde{y}_1,x) d \lambda(\tilde{y}_1) \cdot \int  1_{A_y}(\tilde{y}_2)f_2(\tilde{y}_2,x) d \lambda(\tilde{y}_2) }
d \mu(x) \right )\\
& \leq 2\left(1- \sum_y \int \int  1_{A_y}(\tilde{y}) \sqrt{f_1(\tilde{y},x) f_2(\tilde{y},x)  } d \lambda(\tilde{y})
d \mu(x) \right ) = d_{H}(f_1,f_2).
\end{align*}
For fixed $(y,x)$,
\[
\int_{A_y} (f_1(\tilde{y},x)/p_1(y,x)) \log \frac{f_1(\tilde{y},x)/p_1(y,x)}{f_2(\tilde{y},x)/p_2(y,x)} d \lambda(\tilde{y}) \geq 0
\]
since the Kullback-Leibler divergence is nonnegative.  Thus,
\[
\int_{A_y} f_1(\tilde{y},x) \log \frac{f_1(\tilde{y},x)}{f_2(\tilde{y},x)} d \lambda(\tilde{y}) \geq 
\int_{A_y} f_1(\tilde{y},x) \log \frac{p_1(y,x)}{p_2(y,x)} d \lambda(\tilde{y}) = p_1(y,x) \log \frac{p_1(y,x)}{p_2(y,x)} .
\]
This inequality integrated with respect to $d\mu(x)$ and summed over $y$ implies \eqref{eq:dKLp_dKLf}.
The last claim follows from 
\[
 f_2(\tilde{y},x) \inf_{\tilde{z} \in A_y} \frac{f_1(\tilde{z},x)}{f_2(\tilde{z},x)} \leq 
f_1(\tilde{y},x) \leq f_2(\tilde{y},x)
\sup_{\tilde{z} \in A_y} \frac{f_1(\tilde{z},x)}{f_2(\tilde{z},x)}.
\]
\end{proof}

%\begin{lemma} 
%\label{lm:hi_geq_8Ni}
%For $i \in J_\ast^c$, $h_i\geq8/N_i$.
%\end{lemma}
%\begin{proof} %(Lemma \ref{lm:hi_geq_8Ni})
%By definition of $\Gamma_n$,
%\[
%\left[ \frac{N_{J_\ast} N_i }{n} \right ]^{\frac{1}{2+\beta_{J^c_\ast}^{-1}-\beta_i^{-1}}}	
%\geq
%\left[ \frac{N_{J_\ast}}{n} \right ]^{\frac{1}{2+\beta_{J^c_\ast}^{-1}}} 
%\]
%\[
%\Rightarrow N_i \geq \left[ \frac{N_{J_\ast}}{n} \right ]^{\frac{2+\beta_{J^c_\ast}^{-1} 
%- \beta_i^{-1}}{2+\beta_{J^c_\ast}^{-1}}} = \Gamma_n^{-\beta_i^{-1}}=8 h_i^{-1}
%\]
%\end{proof}

\begin{lemma} 
\label{lm:b_bast_h_N}
For $\Gamma_n$, $h_i$, $\varrho_i$, and $\beta_i^\ast$ defined in Section \ref{sec:main_results},
(i) $\beta_i^\ast \geq \beta_i$ for $i=1,\ldots,d$ and 
%(ii) $h_i\geq 8/N_i$, $i=1,\ldots,d$.
(ii) $\varrho_i\in(1,2]$ for $i \in J_\ast^c \cap \{1,\ldots,d_y\}$.
\end{lemma}

\begin{proof} %(Lemma \ref{lm:q_tvd})
For $i \notin J_\ast$, $\beta_i^\ast = \beta_i$ by definition.
For $i \in J_\ast$, from the definition of $\Gamma_n$, 
\[
\Gamma_n \leq 
\left[ \frac{N_{J_\ast}/N_i }{n} \right ]^{\frac{1}{2+\beta_{J^c_\ast}^{-1}+\beta_i^{-1}}}	
= \Gamma_n^{\frac{2+\beta_{J^c_\ast}^{-1}}{2+\beta_{J^c_\ast}^{-1}+\beta_i^{-1}}}
N_i^{\frac{-1}{2+\beta_{J^c_\ast}^{-1}+\beta_i^{-1}}},
\]
which implies $N_i^{-\beta_i} \geq \Gamma_n$. By the definition of $\beta_i^\ast$, $N_i^{-\beta_i^\ast}=\Gamma_n$ and, thus, $\beta_i^\ast \geq \beta_i$.

%For $i \in J_\ast$, $h_i = 8/N_i$ by definition.
%For $i \in J_\ast^c$, from the definition of $\Gamma_n$,
%\[
%\left[ \frac{N_{J_\ast} N_i }{n} \right ]^{\frac{1}{2+\beta_{J^c_\ast}^{-1}-\beta_i^{-1}}}	
%\geq
%\left[ \frac{N_{J_\ast}}{n} \right ]^{\frac{1}{2+\beta_{J^c_\ast}^{-1}}},
%\]
%which implies 
%\[
%N_i \geq \left[ \frac{N_{J_\ast}}{n} \right ]^{\frac{2+\beta_{J^c_\ast}^{-1} 
%- \beta_i^{-1}}{2+\beta_{J^c_\ast}^{-1}}} = \Gamma_n^{-\beta_i^{-1}}=8 h_i^{-1}
%\]
%and the claimed result (ii). 
%\color{red}
%For $i \in J_\ast$, $h_i = 2/N_i$ by definition so $R_i=1$.
For $i \in J_\ast^c$, from the definition of $\Gamma_n$,
\[
\left[ \frac{N_{J_\ast} N_i }{n} \right ]^{\frac{1}{2+\beta_{J^c_\ast}^{-1}-\beta_i^{-1}}}	
\geq
\left[ \frac{N_{J_\ast}}{n} \right ]^{\frac{1}{2+\beta_{J^c_\ast}^{-1}}},
\]
which implies 
\begin{align*}
N_i &\geq \left[ \frac{N_{J_\ast}}{n} \right ]^{\frac{2+\beta_{J^c_\ast}^{-1} 
- \beta_i^{-1}}{2+\beta_{J^c_\ast}^{-1}}} = \Gamma_n^{-\beta_i^{-1}} \implies \\
  \Gamma_n^{\beta_i^{-1}} & \geq  \frac{1}{N_i},
\end{align*}
and, therefore,  $\Gamma_n^{\beta_i^{-1}}N_i\geq 1$.  Next, define 
\begin{align*}
\varrho_i=\frac{\left\lfloor{ \Gamma_n^{\beta_i^{-1}}N_i/2}\right\rfloor+1}{\Gamma_n^{\beta_i^{-1}}N_i/2}.
\end{align*}
 Then $\varrho_i\in(1,2]$ as $\Gamma_n^{\beta_i^{-1}}N_i\geq 1$.
 %and  $h_i = \varrho_i \Gamma_n^{\beta_i^{-1}} =2/N_i\star \left( \left\lfloor{ \Gamma_n^{\beta_i^{-1}}N_i/2}\right\rfloor+1\right)$ and the claimed result (ii) holds with $R_i=\left\lfloor{ \Gamma_n^{\beta_i^{-1}}N_i/2}\right\rfloor+1$. 
\end{proof}

\subsection{Proofs and Auxiliary Results for Posterior Contraction Rates} 
\label{sec:sieve_entropy}

\subsubsection{Prior Thickness} 
\label{sec:app_prior_thickness}

\begin{lemma} 
\label{lm:AnisoTaylor} (Anisotropic Taylor Expansion)
For $f \in \mathcal{C}^{\beta_1,\ldots,\beta_d,L}$ and $r \in \{1,\ldots,d\}$
\begin{align}
\label{eq:anistaylor1}
f(x_1+y_1,\ldots,x_d+y_d) =  \sum_{k \in I^r} \frac{y^k}{k!} & D^k f (x_1,\ldots,x_r,x_{r+1}+y_{r+1},\ldots,x_{d}+y_{d}) \\
\label{eq:anistaylor2}
 + \sum_{l=1}^r  \sum_{k \in \bar{I}^l} \frac{y^k}{k!} 
 \biggl(&  
D^k f (x_1,\ldots,x_l+\zeta_l^{k},x_{l+1}+y_{l+1},\ldots,x_{d}+y_{d}) \\
\label{eq:anistaylor3}
& -
D^k f (x_1,\ldots,x_l,x_{l+1}+y_{l+1},\ldots,x_{d}+y_{d})
\biggr),
\end{align}
where $\zeta_l^{k} \in [x_l,x_l+y_l]\cup [x_l+y_l,x_l]$,
\begin{align*}
I^l & = \biggl \{k=(k_1,\ldots,k_l,0,\ldots,0) \in \mathbb{Z}_+^d: \: k_i \leq \bigl \lfloor \beta_i (1-\sum_{j=1}^{i-1} k_j/\beta_j) \bigr \rfloor, \, i=1,\ldots,l \biggr \},
\\
\bar{I}^l & = \biggl\{k \in I^l: \; k_l = \bigl \lfloor \beta_l (1-\sum_{j=1}^{l-1} k_j/\beta_j) \bigr \rfloor \biggr\}, 
\end{align*}
and the differences in derivatives in \eqref{eq:anistaylor2}-\eqref{eq:anistaylor3} are bounded by
$
L \left|\zeta_l^{k}\right|^{\beta_l(1-\sum_{i=1}^d k_i/\beta_i)}.
$
\end{lemma}
\begin{proof} %(Lemma \ref{lm:AnisoTaylor})
The lemma is proved by induction.  For $r=1$, \eqref{eq:anistaylor1}-\eqref{eq:anistaylor3} is a standard univariate Taylor expansion of $f(x+y)$ in the first argument around $(x_1,x_2+y_2,\ldots,x_d+y_d)$. 
Suppose \eqref{eq:anistaylor1}-\eqref{eq:anistaylor3} holds for some $r \in \{1,\ldots,d\}$.
Then, let us show that \eqref{eq:anistaylor1}-\eqref{eq:anistaylor3} holds for $r+1$.
For that, consider a univariate Taylor expansion of $D^k f$ in \eqref{eq:anistaylor1}.  
The following notation will be useful. Let $e_i \in \mathbb{R}^d$, $i=1,\ldots,d$, 
be such that $e_{ij}=1$ for $i=j$ and $e_{ij}=0$ for $i\neq j$ and $k_{r+1}^\ast = \lfloor \beta_{r+1} (1-\sum_{j=1}^{r} k_j/\beta_j)  \rfloor$.  Then,
\begin{align*}
%\label{eq:anistaylorDk1}
& D^k f (x_1,\ldots,x_r,x_{r+1}+y_{r+1},\ldots,x_{d}+y_{d}) =  \\
& \sum_{k_{r+1}=0}^{k_{r+1}^\ast} 
\frac{y_{r+1}^{k_{r+1}}}{k_{r+1}!}  D^{k+k_{r+1}\cdot e_{r+1}} f (x_1,\ldots,x_{r+1},x_{r+2}+y_{r+2},\ldots,x_{d}+y_{d}) 
\\
%\label{eq:anistaylorDk2}
 & +   \frac{y_{r+1}^{k_{r+1}^\ast}}{k_{r+1}^\ast!} 
 \biggl(  
D^{k+k_{r+1}^\ast\cdot e_{r+1}} f (x_1,\ldots,x_r, x_{r+1}+\zeta_{r+1}^{{k+k_{r+1}^\ast\cdot e_{r+1}}},x_{r+2}+y_{l+2},\ldots,x_{d}+y_{d}) \\
%\label{eq:anistaylorDk3}
& -
D^{k+k_{r+1}^\ast\cdot e_{r+1}} f (x_1,\ldots,x_r, x_{r+1},x_{r+2}+y_{l+2},\ldots,x_{d}+y_{d})
\biggr).
\end{align*}
Inserting this expansion into \eqref{eq:anistaylor1} delivers the result for $r+1$.

\end{proof}

\begin{lemma} 
\label{lm:AnisoTaylorRemainderBd}
Let $R(x,y)$ denote the remainder term in the anisotropic Taylor expansion 
(\eqref{eq:anistaylor2}-\eqref{eq:anistaylor3} for $r=d$).
Suppose $f \in \mathcal{C}^{\beta_1,\ldots,\beta_d,L}$ and $L$ satisfies  \eqref{eq:L_defb}-\eqref{eq:tildeL_ineq}.  
Let $\sigma = \{\sigma_i=\sigma_{n}^{\beta/\beta_i}, \: i=1,\ldots,d \}$ and $\sigma_n \rightarrow 0$.
Then, for all sufficiently large $n$,
\[
\int |R(x,y)| \phi (y; 0,\sigma) dy \lesssim {L}(x) \sigma_n^\beta.
\]
\end{lemma}

\begin{proof}
Note that $|R(x,y)|$ is bounded by a sum of the following terms over $k\in \bar{I}^l$ and $l \in \{1,\ldots,d\}$
\begin{align*}
& \frac{y^k}{k!} 
 \biggl|
D^k f (x_1,\ldots,x_l+\zeta_l^{k},x_{l+1}+y_{l+1},\ldots,x_{d}+y_{d}) -
D^k f (x_1,\ldots,x_l,x_{l+1}+y_{l+1},\ldots,x_{d}+y_{d})
\biggr| \\
& \leq \frac{y^k}{k!} L\left(x+(0,\ldots,0,y_{l+1:d}),\zeta_l^{k} e_l \right ) \left|\zeta_l^{k} \right|^{\beta_l(1-\sum_{i=1}^d k_i/\beta_i)} 
\\
&  
\leq \tilde{L}(x) \exp\left\{\tau_0 || y_{l+1:d}||^2 \right\} 
\exp\left\{\tau_0 || \zeta_l^{k}||^2 \right\} \left|\zeta_l^{k} \right|^{\beta_l(1-\sum_{i=1}^d k_i/\beta_i)}
\\ &  
\leq \tilde{L}(x)  \frac{y^k}{k!}\exp\left\{\tau_0 || y||^2 \right\} \left|y_l \right|^{\beta_l(1-\sum_{i=1}^d k_i/\beta_i)},
\end{align*}
where we used inequalities \eqref{eq:smooth_def}, \eqref{eq:L_defb}, and \eqref{eq:tildeL_ineq} and that $\left|\zeta_l^{k} \right|\leq \left|y_l \right|$.

For all sufficiently large $n$ such that $\tau_0< 0.5/\max_i \sigma_i^2$, 
\begin{align*}
&\int \left| 
\tilde{L}(x)  \frac{y^k}{k!}\exp\left\{\tau_0 || y||^2 \right\} \left|y_l \right|^{\beta_l(1-\sum_{i=1}^d k_i/\beta_i)}  
\right| \phi(y;0,\sigma) dy  \\
&\lesssim \tilde{L}(x)  
\prod_{i=1}^{{l-1}} \int |y_i|^{k_i}
\phi(y_i;0;\sigma_i \sqrt{2}) dy_i
\cdot
\int
y_l^{k_l} \left|y_l \right|^{\beta_l(1-\sum_{i=1}^d k_i/\beta_i)}  
\phi(y_l;0;\sigma_l \sqrt{2}) dy_l
\\
&\lesssim  \tilde{L}(x) \sigma_1^{k_1} \cdots \sigma_{l-1}^{k_{l-1}}
\sigma_{l}^{k_l+\beta_l(1-\sum_{i=1}^d k_i/\beta_i)}  
\\&
=  \tilde{L}(x) \sigma_n^{k_1 \beta/\beta_1} \cdots \sigma_n^{k_l \beta/\beta_l}\sigma_n^{\frac{\beta}{\beta_l}\beta_l(1-\sum_{i=1}^d k_i/\beta_i)} =\tilde{L}(x) K_2 \sigma_n^\beta,
\end{align*}
where we use  $\int |z|^\rho \phi(z,0,\omega) dz \lesssim \omega^\rho$
and $k_{l+1}=\cdots=k_d=0$ for $k \in \bar{I}_l$.
Thus, the claim of the lemma follows.
% as necessary and the argument on p. 637 of \cite{ShenTokdarGhosal2013} goes through.

\end{proof}

\begin{lemma} 
\label{lm:finCb_fcondJinC} 
Suppose density $f_0 \in \mathcal{C}^{\beta_1,\ldots,\beta_d,L}$ with a constant envelope $L$ has support on $[0,1]^d$ and $f_0(z)\geq \underline{f}>0$.
Then, $f_{0|J} \in \mathcal{C}^{\beta_{d_{J^c}},\ldots,\beta_d,L/\underline{f}}$.
\end{lemma}
\begin{proof} 
For $\tilde{x}, \Delta \tilde{x} \in \mathcal{X}$, $y_J \in \mathcal{Y}_J$, and some $\tilde{y}_{J}^\ast \in A_{y_J}$, by the mean value theorem,
\begin{align*}
& D^{k}f_{0|J}(\tilde{x}+\Delta \tilde{x}|y_J)-D^{k}f_{0|J}(\tilde{x}|y_J) = 
\\
&
=\frac{1}{\pi_{0J}(y_J)} \int_{A_{y_J} } \left (
D^{0,\ldots,0,k}f_{0}(\tilde{y}_J, \tilde{x}+\Delta \tilde{x})-D^{0,\ldots,0,k}f_{0}(\tilde{y}_J,\tilde{x}) \right )
 d\tilde{y}_J
\\
&
= \frac{1/N_J}{\pi_{0J}(y_J)} 
\left (D^{0,\ldots,0,k}f_{0}(\tilde{y}_J^\ast, \tilde{x}+\Delta \tilde{x})-D^{0,\ldots,0,k}f_{0}(\tilde{y}_J^\ast,\tilde{x}) \right)
\end{align*}
and the  claim of the lemma follows from the definition of $\mathcal{C}^{\beta_1,\ldots,\beta_d,L}$ 
and $\pi_{0J}(y_J) \geq \underline{f}/N_J$.
\end{proof} 

\begin{lemma}
\label{lm:dH_KL}
There is a $\lambda_0 \in (0,1)$ such that for any $\lambda \in (0,\lambda_0)$ and any two conditional densities $p,q \in \mathcal{F}$, a probability measure $P$ on $\mathcal{Z}$ that has a conditional density equal to $p$, and $d_h$ defined with the distribution on $\mathcal{X}$ implied by $P$,
\[
P \log \frac{p}{q} \leq d_h^2(p,q) \left( 1+ 2 \log\frac{1}{\lambda}\right) 
+ 2 P \left \{  \left(\log \frac{p}{q} \right) 1 \left( \frac{q}{p}\leq \lambda \right) \right\},
\]
\[
P \left(\log \frac{p}{q} \right)^2 \leq d_h^2(p,q) \left( 12+ 2 \left(\log\frac{1}{\lambda}\right)^2\right) 
+ 8 P \left \{  \left(\log \frac{p}{q} \right)^2 1 \left( \frac{q}{p}\leq \lambda \right) \right\},
\]
\end{lemma}

\begin{proof}
The proof is exactly the same as the proof of Lemma 4 of \cite{ShenTokdarGhosal2013}, which in turn, follows the proof of Lemma 7 in \cite{GhosalVandervaart:07}.
\end{proof}

\begin{lemma}
\label{lm:ProbBoundOutside}
Under the assumptions and notation of Section \ref{sec:post_rates}, for for some $B_0 \in (0,\infty)$ and any $y_J\in\mathcal{Y}_J$,
\begin{align*}
F_{0|J}\left(||\tilde{X}|| > a_{\sigma_n} \big | y_J\right) \leq B_0 \sigma_n^{4\beta+2\varepsilon}\underline{\sigma}_{n} ^8.
\end{align*}

\end{lemma}
\begin{proof}
Note that in the proof of Proposition 1 of \cite{ShenTokdarGhosal2013} it is shown that $a_{\sigma_n}^{STG} >a$, where $a_0^{STG} = \{ (8\beta + 4\varepsilon +16)/(b \delta)\}^{1/\tau}$ and  $a_{\sigma_n}^{STG}=a_0^{STG} \log(1/\sigma_n)^{1/\tau}$. As $a_0>a_0^{STG}$ and $a_{\sigma_n}>a_{\sigma_n}^{STG}$, therefore $a_{\sigma_n}>a$. Define $E_{\sigma_n}^* = \left\{
\tilde{x}\in\mathbb{R}^{d_{J^c}}\:  :\:  f_{0|J}(\tilde{x}|y_J) \geq  \sigma_n^{(4\beta+2\varepsilon + 8 \beta/\beta_{\min})/\delta}
 \right\}$.
Note that by construction of $s_2$ in proof of Proposition 1 of \cite{ShenTokdarGhosal2013} and as $\sigma_n < s_2$ it follows that 
\begin{align*}
\frac{(4\beta+2\varepsilon+8)}{b\delta} \log \left( \frac{1}{\sigma_n} \right) \geq \frac{1}{b} \log \bar{f}_0 \implies
\sigma_n^{-\frac{(4\beta+2\varepsilon+8)}{\delta}} \geq \bar{f}_0 .
\end{align*}
For $\tilde{x}\in E_{\sigma_n}^*$,
\begin{align*}
f_{0|J}(\tilde{x}|y_J) &\geq  \sigma_n^{(4\beta+2\varepsilon + 8 \beta/\beta_{\min})/\delta} = \sigma_n^{(8\beta+4\varepsilon + 8 \beta/\beta_{\min}+8)/\delta} \sigma_n^{-(4\beta+2\varepsilon +8)/\delta} \\
& \geq \bar{f}_0 \sigma_n^{(8\beta+4\varepsilon + 8 \beta/\beta_{\min}+8)/\delta} = \bar{f}_0 \sigma_n^{a_0^\tau b} =\bar{f}_0 \exp\left\{ -b a_0^\tau \log(\frac{1}{\sigma_n}) \right\}\\
&= \bar{f}_0 \exp\left\{ -b \left(a_0 (\log(\frac{1}{\sigma_n})^{1/\tau})\right)^\tau \right\} = \bar{f}_0 \exp\left\{-b a_{\sigma_n}^\tau\right\}.
\end{align*}
As $a_{\sigma_n}>a$ and as $f_{0|J}(\tilde{x}|y_J) \geq \bar{f}_0 \exp\{-b a_{\sigma_n}^\tau\}$, then the tail condition (\ref{eq:tail_cond2}) is satisfied only if $||\tilde{x}|| < a_{\sigma_n}$. Therefore, $E_{\sigma_n}^*  \subset \left\{ \tilde{x}\in \mathbb{R}^{d_J} \: : \: ||\tilde{x}||\leq a_{\sigma_n}\right\}$. As in the proof of Proposition 1 of \cite{ShenTokdarGhosal2013}, by Markov's inequality,
\begin{align*}
F_{0|J}\left(||\tilde{X}|| > a_{\sigma_n}|y_J\right)&\leq F_{0|J}(E_{\sigma_n}^{*,c}|y_J) =   F_{0|J}\left(f_{0|J}(\tilde{x}|y_J)^{-\delta}>\sigma_n^{-(4\beta+2\varepsilon + 8 \beta/\beta_{\min})}|y_J\right) \\
&\leq B_0 \sigma_n^{4\beta+2\varepsilon+8\beta/\beta_{\min}}
= B_0 \sigma_n^{4\beta+2\varepsilon}\underline{\sigma}_{n} ^8
\end{align*}
as desired since $\sigma_n^{\beta/\beta_{\min}}=\underline{\sigma}_n$ and the tail condition on $f_{0|J}(\cdot|y_J)$, \eqref{eq:tail_cond2}, implies the existence of a $\delta>0$ small enough such that $E_{0|J}(f_{0|J}^{-\delta}) \leq B_0<\infty$ for any $y_J \in \mathcal{Y}_J$.
\end{proof}

\begin{lemma}
\label{lm:H_dist_model_pi0_model}
Under the assumptions and notation of Section \ref{sec:post_rates},
for $m=K N_J$ and any $\theta\in S_{\theta^\star}$
\[
d_H^2(p_{|J}^\star(\cdot |\cdot)\pi_0(\cdot), p(\cdot,\cdot | \theta, m )) \lesssim \sigma_n^{2\beta}.
\]
\end{lemma}

\begin{proof}

Let us define
\[
f_{J}(y_J,\tilde{x}| \theta, m)= \int_{A_{y_J}} f(\tilde{y}_J, \tilde{x}| \theta, m) d\tilde{y}_J.
\]
Then, 
\begin{align*}
&d_H^2(p_{|J}^\star(\cdot |\cdot)\pi_0(\cdot), p(\cdot,\cdot | \theta, m )) \leq d_{TV}(p_{|J}^\star(\cdot |\cdot)\pi_0(\cdot), p(\cdot,\cdot | \theta, m ))\\
&\leq d_{TV}(f_{|J}^\star(\cdot |\cdot)\pi_0(\cdot), f_{J}(\cdot,\cdot | \theta, m )) \\
&= \sum_{y_J\in\mathcal{Y}_J} \int_{\tilde{\mathcal{X}}} \bigg|
\sum_{k\in\mathcal{Y}_J} \sum_{j=1}^{K}  \alpha_{j|k}^\star \pi_0(k) 
\indic
\{k=y_J\} \phi(\tilde{x},\mu_{j|k}^{\star},\sigma_{J^c}^\star) 
\\& \;\;\;\;\;\;\;\;\;\;\;\;\;\;\;\;\;\;\;\;\;\;
-
\alpha_{jk}\int_{A_{y_J}} \phi(\tilde{y}_J,\mu_{jk,J},\sigma_J) d\tilde{y}_J\cdot \phi(\tilde{x},\mu_{jk,J^c},\sigma_{J^c}) 
  \bigg| d\tilde{x}\\
%&=
%\sum_{y_J\in\mathcal{Y}_J} \int_{\tilde{\mathcal{X}}} \left|
%\sum_{k\in\mathcal{Y}_J} \sum_{j=1}^{K} \alpha_{jk}^\star
%\indic
%\{k=y_J\} \phi(\tilde{x},\mu_{jk}^{x\star},\Sigma_x^\star) - 
%\alpha_{jk}^\star\indic \{k=y_J\} \phi(\tilde{x},\mu_{jk}^x,\Sigma_x) 
%\right. 
%\\
%&\left.
%+ 
%\alpha_{jk}^\star\indic \{k=y_J\} \phi(\tilde{x},\mu_{jk}^x,\Sigma_x) 
%-
%\alpha_{jk}\int_{A_{y_J}} \phi(\tilde{y},\mu_{jk}^y,\Sigma_y) d\tilde{y}
%\phi(\tilde{x},\mu_{jk}^x,\Sigma_x) 
%\right| d\tilde{x}\\
&\leq
\sum_{y_J\in\mathcal{Y}_J} \int_{\tilde{\mathcal{X}}} \left|
\sum_{k\in\mathcal{Y}_J} \sum_{j=1}^{K} \alpha_{jk}^\star
\indic
\{k=y_J\} \phi(\tilde{x},\mu_{j|k}^{\star},\sigma_{J^c}^\star) - 
\alpha_{jk}^\star\indic \{k=y_J\} \phi(\tilde{x},\mu_{jk,J^c},\sigma_{J^c}) 
\right| d\tilde{x}
\\
&
+ 
\sum_{y_J\in\mathcal{Y}_J} \int_{\tilde{\mathcal{X}}} \left|
\sum_{k\in\mathcal{Y}_J} \sum_{j=1}^{K}
\alpha_{jk}^\star\indic \{k=y_J\} \phi(\tilde{x},\mu_{jk,J^c},\sigma_{J^c}) 
-
\alpha_{jk}\int_{A_{y_J}} \phi(\tilde{y},\mu_{jk,J},\sigma_J) d\tilde{y}
\phi(\tilde{x},\mu_{jk,J^c},\sigma_{J^c}) 
\right| d\tilde{x},
\end{align*}
where the first inequality follows from $d_H^2(\cdot,\cdot)\leq d_{TV}(\cdot,\cdot)$, the second inequality holds by Lemma \ref{lm:p_f_distance_bds}, and the last inequality is obtained 
by 
%adding and subtracting $\alpha_{jk}^\star\indic \{k=y_J\} \phi(\tilde{x},\mu_{jk,J^c},\sigma_{J^c})$ and 
the triangle inequality.

Let's explore the two parts of the right hand side in the last inequality independently. 
First,
\begin{align*}
&\sum_{y_J\in\mathcal{Y}_J} \int_{\tilde{\mathcal{X}}} \left|
\sum_{k\in\mathcal{Y}_J} \sum_{j=1}^{K} \alpha_{jk}^\star
\indic
\{k=y_J\} \phi(\tilde{x},\mu_{j|k}^{\star},\sigma_{J^c}^\star) - 
\alpha_{jk}^\star\indic \{k=y_J\} \phi(\tilde{x},\mu_{jk,J^c},\sigma_{J^c}) 
\right| d\tilde{x}\\
&\leq
 \sum_{y_J\in\mathcal{Y}_J}
\sum_{k\in\mathcal{Y}_J} \sum_{j=1}^{K} \alpha_{jk}^\star
\indic\{k=y_J\}
 \int_{\tilde{\mathcal{X}}} \left| \phi(\tilde{x},\mu_{j|k}^{\star},\sigma_{J^c}^\star) - \phi(\tilde{x},\mu_{jk,J^c},\sigma_{J^c}) 
\right| d\tilde{x}
\\
&\leq 
\max_{j\leq N,k\in\mathcal{Y}_J} d_{TV}(\phi(\cdot; \mu_{j|k}^{\star},\sigma_{J^c}^\star),  \phi(\cdot,\mu_{jk,J^c},\sigma_{J^c})) \lesssim \sigma_n^{2\beta},
\end{align*}
where the fact that $\alpha^\star_{j,k}=0$ for $j>N$ by design is used to get $j\leq N$ rather than $j\leq K$ in the max subscript.
The last inequality is proved in Lemma \ref{lm:TV_normal_dist}.

Second,
\begin{align*}
&\sum_{y_J\in\mathcal{Y}_J} \int_{\tilde{\mathcal{X}}} \left|
\sum_{k\in\mathcal{Y}_J} \sum_{j=1}^{K} \alpha_{jk}^\star
\indic \{k=y_J\} \phi(\tilde{x},\mu_{jk,J^c},\sigma_{J^c}) 
-
\alpha_{jk}\int_{A_{y_J}} \phi(\tilde{y}_J,\mu_{jk,J},\sigma_J) d\tilde{y}_J\phi(\tilde{x},\mu_{jk,J^c},\sigma_{J^c}) 
\right| d\tilde{x}\\
&=
\sum_{j=1}^{K} 
\left( 
\sum_{y_J\in\mathcal{Y}_J}  \left|
\sum_{k\in\mathcal{Y}_J}
\alpha_{jk}^\star\indic \{k=y_J\} 
-
\alpha_{jk}\int_{A_{y_J}} \phi(\tilde{y}_J,\mu_{jk,J},\sigma_J) d\tilde{y}_J
\right| 
\int_{\tilde{\mathcal{X}}} \phi(\tilde{x},\mu_{jk,J^c},\sigma_{J^c})  d\tilde{x}
\right)\\
&=
\sum_{j=1}^{K} 
\sum_{y_J\in\mathcal{Y}_J}  \left|
\sum_{k\in\mathcal{Y}_J}
\alpha_{jk}^\star\indic \{k=y_J\} 
-
\alpha_{jk}\int_{A_{y_J}} \phi(\tilde{y}_J,\mu_{jk,J},\sigma_J) d\tilde{y}_J
\right| 
\\
&\leq 
\sum_{y_J\in\mathcal{Y}_J} 
\sum_{k\in\mathcal{Y}_J}\sum_{j=1}^{K} 
\left|
\alpha_{jk}^\star\indic \{k=y_J\} 
-
\alpha_{jk}^\star \int_{A_{y_J}} \phi(\tilde{y}_J,\mu_{jk,J},\sigma_J) d\tilde{y}_J
\right| 
\\
&
+
\sum_{y_J\in\mathcal{Y}_J} 
\sum_{k\in\mathcal{Y}_J}\sum_{j=1}^{K} 
\left|
\alpha_{jk}^\star \int_{A_{y_J}} \phi(\tilde{y}_J,\mu_{jk,J},\sigma_J) d\tilde{y}_J
-
\alpha_{jk} \int_{A_{y_J}} \phi(\tilde{y}_J,\mu_{jk,J},\sigma_J) d\tilde{y}_J
\right| \\
&\leq 
\sum_{y_J\in\mathcal{Y}_J} 
\sum_{k\in\mathcal{Y}_J}\sum_{j=1}^{K} 
\alpha_{jk}^\star
\left|
\indic \{k=y_J\} 
- \int_{A_{y_J}} \phi(\tilde{y}_J,\mu_{jk,J},\sigma_J) d\tilde{y}_J
\right|\\
&+
\sum_{y_J\in\mathcal{Y}_J} 
\sum_{k\in\mathcal{Y}_J}\sum_{j=1}^{K} 
\left|
\alpha_{jk}^\star
-
\alpha_{jk} 
\right| 
\int_{A_{y_J}} \phi(\tilde{y}_J,\mu_{jk,J},\sigma_J) d\tilde{y}_J\\
&= 
\sum_{k\in\mathcal{Y}_J}\sum_{j=1}^{K} 
\left(
\alpha_{jk}^\star
\sum_{y_J\in\mathcal{Y}_J} 
\left|
\indic \{k=y_J\} 
- \int_{A_{y_J}} \phi(\tilde{y}_J,\mu_{jk,J},\sigma_J) d\tilde{y}_J
\right|
\right)\\
&+
\sum_{k\in\mathcal{Y}_J}\sum_{j=1}^{K} 
\left( 
\left|
\alpha_{jk}^\star
-
\alpha_{jk} 
\right|
\sum_{y_J\in\mathcal{Y}_J} 
\int_{A_{y_J}} \phi(\tilde{y}_J,\mu_{jk,J},\sigma_J) d\tilde{y}_J
\right)
 \\
 &\leq 
\sum_{k\in\mathcal{Y}_J}\sum_{j=1}^{K} 
\alpha_{jk}^\star
\left[
\int_{A^c_{k}}\phi(\tilde{y}_J,\mu_{jk,J},\sigma_J) d\tilde{y}_J
+
\sum_{y_J\neq k}  \int_{A_{y_J}} \phi(\tilde{y}_J,\mu_{jk,J},\sigma_J) d\tilde{y}_J
\right]
+
\sum_{k\in\mathcal{Y}_J}\sum_{j=1}^{K} 
\left|
\alpha_{jk}^\star
-
\alpha_{jk} 
\right|\\
 &= 
\sum_{k\in\mathcal{Y}_J}\sum_{j=1}^{K} 
\alpha_{jk}^\star
\cdot 2 
\int_{A^c_{k}}\phi(\tilde{y}_J,\mu_{jk,J},\sigma_J) d\tilde{y}_J
+
\sum_{k\in\mathcal{Y}_J}\sum_{j=1}^{K} 
\left|
\alpha_{jk}^\star
-
\alpha_{jk} 
\right|\\
& \leq
2 \max_{j\leq N,k\in\mathcal{Y}_J} \int_{A^c_{k}}\phi(\tilde{y}_J,\mu_{jk,J},\sigma_J) d\tilde{y}_J
+
\sum_{k\in\mathcal{Y}_J}\sum_{j=1}^{K} 
\left|
\alpha_{jk}^\star
-
\alpha_{jk} 
\right| \lesssim \sigma_n^{2\beta}.
\end{align*}
The last inequality follows from Lemma \ref{lm:integral_approx_Ak} and the definition of $S_{\theta^\star}$.

\end{proof}

\begin{lemma}
\label{lm:TV_normal_dist}
Under the assumptions and notation of Section \ref{sec:post_rates},
\begin{align*}
&
\max_{j\leq N,k\in\mathcal{Y}_J} d_{TV}(\phi(\cdot; \mu_{j|k}^{\star},\sigma_{J^c}^\star),  
\phi(\cdot,\mu_{jk,J^c},\sigma_{J^c}))
\lesssim \sigma_n^{2\beta}.
\end{align*}

\end{lemma}
\begin{proof}
Fix some $j\leq N$ and $k\in\mathcal{Y}_J$.
It is known that 
\[
d_{TV}(\phi(\cdot; \mu_{j|k}^{\star},\sigma_{J^c}^\star),  \phi(\cdot,\mu_{jk,J^c},\sigma_{J^c})) \leq 2 \sqrt{d_{KL}(\phi(\cdot; \mu_{j|k}^{\star},\sigma_{J^c}^\star),  \phi(\cdot,\mu_{jk,J^c},\sigma_{J^c}))}
\]
and
\[
d_{KL}(\phi(\cdot; \mu_{j|k}^{\star},\sigma_{J^c}^\star),  \phi(\cdot,\mu_{jk,J^c},\sigma_{J^c}))=
\sum_{i \in J^c} \frac{ \sigma_{i}^2}{\sigma_{i}^{\star2}} -1 - \log \frac{ \sigma_{i}^2}{\sigma_{i}^{\star2}} +  \frac{(\mu_{j|k,i}^{\star}-\mu_{jk,i})^2}{\sigma_{i}^{\star2}}.
\]
From the definition of $S_{\theta^\star}$, 
\[
\sum_{i \in J^c}  \frac{(\mu_{j|k,i}^{\star}-\mu_{jk,i})^2}{\sigma_{i}^{\star2}} \leq \tilde{\epsilon}_n^{4 b_1} \leq \sigma_n^{4\beta}.
\]
Since $\sigma_{i}^2\in (\sigma_{i}^{\star2}(1+\sigma_n^{2\beta})^\inv, \sigma_{i}^{\star2})$
and 
the  fact that $|z -1-\log z| \lesssim |z-1|^2$ for $z$ in a neighborhood of $1$, we have for all sufficiently large $n$
\[
\left |\frac{ \sigma_{i}^2}{\sigma_{i}^{\star2}} -1 - \log \frac{ \sigma_{i}^2}{\sigma_{i}^{\star2}}\right|
\lesssim \left(1-\frac{ \sigma_{i}^2}{\sigma_{i}^{\star2}}\right)^2
\lesssim \sigma_n^{4\beta}.
\]
The three inequalities derived above imply the claim of the lemma.

\end{proof}

\begin{lemma}
\label{lm:integral_approx_Ak}
Under the assumptions and notation of Section \ref{sec:post_rates},
for $\theta \in S_{\theta^\star}$,
\begin{align*}
 \max_{j\leq N,k\in\mathcal{Y}_J} \int_{A^c_{k}}\phi(\tilde{y}_J,\mu_{jk,J},\sigma_J) d\tilde{y}_J \lesssim \sigma_n^{2\beta}.
\end{align*}

\end{lemma}
\begin{proof}
Fix $j\leq N$, $k\in\mathcal{Y}_J$, and $\theta \in S_{\theta^\star}$.
Since $\mu_{jk,i}\in\left[ {k_i}-\frac{1}{4N_i},{k_i}+\frac{1}{4N_i}
  \right]$,
\begin{align*}
 \int_{A^c_{k}}\phi(\tilde{y}_J,\mu_{jk,J},\sigma_J) d\tilde{y}_J &\leq \sum_{i \in J} Pr \left( \tilde{y}_i \notin \left[
k_i-\frac{1}{2N_i},k_i+\frac{1}{2N_i}
  \right]
 \right) \\
 & \leq  \sum_{i \in J} Pr \left( \tilde{y}_i \notin \left[
\mu_{jk,i}-\frac{1}{4N_i},\mu_{jk,i}+\frac{1}{4N_i}  \right]
 \right) \\
&    = 2  \sum_{i \in J}
      \int_{-\infty}^{-\frac{1}{4N_i\sigma_i}} \phi(\tilde{y}_i,0,1) d\tilde{y}_i\\
&   \leq 
    2  \sum_{i \in J}
\exp\left\{-\frac{1}{2(4N_i\sigma_i)^2}\right\} \leq 2  \sum_{i \in J} \sigma_n^{2\beta} \lesssim \sigma_n^{2\beta},
 \end{align*}
where the last inequality follows from the restrictions on $\sigma_J$ in $S_{\theta^\star}$ and the penultimate inequality follows from  a bound on the normal tail probability derived below.

If $\tilde{Y}_i$ has $N(0,1)$  distribution, then the moment generating function is $M(\theta)=\exp\{\theta^2/2\}$. Note that $\exp\{\theta(\tilde{Y}_i - (4N_i\sigma_i)^\inv ) \}\geq 1$ when $\tilde{Y}_i \leq (4N_i\sigma_i^y)^\inv$ and $\theta\leq 0$, therefore:
\begin{align*}
\int_{-\infty}^{-\frac{1}{4N_i\sigma_i}} \phi(\tilde{y}_i,0,1) d\tilde{y}_i &\leq \inf_{\theta\leq 0} \mathbb{P} \exp\left\{ \theta(\tilde{Y}_i - (4N_i\sigma_i)^\inv )
\right\}=\inf_{\theta\leq 0} \exp\left\{ -\theta (4N_i\sigma_i)^\inv 
\right\}M(\theta)\\
&= \inf_{\theta\leq 0} \exp\left\{ -\theta (4N_i\sigma_i)^\inv 
\right\}\exp\left\{ \theta^2/2
\right\}=\exp\left\{ -(4N_i\sigma_i)^{-2}/2 
\right\}.
\end{align*}
\end{proof}

\begin{lemma}
\label{lm:Bound_Model_truef_ratio}
Under the assumptions and notation of Section \ref{sec:post_rates}, for any 
$(y_J,y_I) \in \mathcal{Y}$,  some constants $C_3, C_4>0$ and 
all sufficiently large $n$,
\begin{align}
\frac{p(y_J,y_I,{x} | \theta, m)}{p_{0}( y_J,y_I,x)} 
\geq
C_3  \frac{\sigma_n^{2\beta }}{m^2}
\equiv \lambda_n,
\end{align}
when $\norm{{x}} \leq a_{\sigma_n}$ and
\begin{align}
\frac{p(y_J,y_I,{x} | \theta, m)}{p_{0}( y_J,y_I,x)} 
 \geq 
\exp\left\{ - \frac{8||{x}||^2}{\underline{\sigma}_{n}^2} - C_4 \log n\right\}
\end{align}
when $\norm{x} > a_{\sigma_n}$.

\end{lemma}
\begin{proof}
By assumption \eqref{eq:tail_cond2}, $f_{0|J}(\tilde{x}|y_J) \leq \bar{f}_0$, and  $\pi_{0J}(y_J)\leq 1$  for all $ (\tilde{x},y_J)$. Therefore,
\begin{align}
\label{eq:dgp_lower_bound}
\frac{f_J(y_J,\tilde{x} | \theta, m)}{f_{0|J}(\tilde{x} | y_J)\pi_{0J}(y_J)} 
\geq \bar{f}_0^\inv f_J(\tilde{x},y_J|\theta,m)
\end{align}

Let $k^\star = y_J$. Then, by Lemma \ref{lm:integral_approx_Ak}, for any $j \in \{1,\ldots,K\}$,
\begin{align*}
\int_{A_{y_J}} \phi(\tilde{y}_J;\mu_{j k^*,J},\sigma_J) d\tilde{y}_J \geq \frac{1}{2}
\end{align*}
for all $n$ large enough as $\sigma_n \rightarrow 0$.

For any $\tilde{x} \in \tilde{\mathcal{{X}}}$  with $\norm{\tilde{x}} \leq 2a_{\sigma_n}$,
by the construction of  sets $U_{j|k^\star}$, there exists $j^\star \in \left\{1,\ldots,K\right\}$  such that $\tilde{x}, \mu_{j^\star | k^\star} \in U_{j^\star|k^\star}$ and for all sufficiently large $n$, $\sum_{i \in J^c}(\tilde{x}_i-\mu_{j^\star | k^\star,i} )^2/ \sigma_i^2 \leq 4$.  Then,
\begin{align*}
\phi(\tilde{x},\mu_{j^\star | k^\star},\sigma_{J^c}) &= (2\pi)^{-d_{J^c}/2} 
\prod_{i \in J^c}\sigma_i^{-1}
 \exp\left\{ 
-0.5 \sum_{i \in J^c}(\tilde{x}_i-\mu_{j^\star | k^\star,i})^2/\sigma_i^2 
\right\} \\
&\geq 
(2\pi)^{-d_{J^c}/2} \sigma_n^{-d_{J^c}} e^{ -2 }
.
\end{align*}
Thus,
\begin{align*}
f_J(y_J,\tilde{x}|\theta)
&=
\sum_{k\in\mathcal{Y}_J}\sum_{j=1}^{K}
\alpha_{jk}\int_{A_{y_J}} \phi(\tilde{y}_J,\mu_{jk,J},\sigma_J) d\tilde{y}_J\phi(\tilde{x},\mu_{jk,J^c},\sigma_{J^c}) \\
&\geq  \alpha_{j^\star k^*}  \phi(\tilde{x},\mu_{j^\star k^\star,J^c},\sigma_{J^c})  \int_{A_{y_J}} \phi(\tilde{y}_J,\mu_{j^\star k^\star,J},\sigma_J) d\tilde{y}_J
\end{align*}
and for $C_{3} = \bar{f}_0^\inv (2\pi)^{-d_{J^c}/2} e^{ -2 }/ 8$,
\begin{align}
\frac{f_J(y_J,\tilde{x} | \theta, m)}{f_{0|J}(\tilde{x} | y_J)\pi_{0J}(y_J)} 
&\geq \bar{f}_0^\inv 
\cdot
\min_{j\leq K, k\in\mathcal{Y}_J}  \alpha_{jk} \cdot
(2\pi)^{-d_{J^c}/2} \sigma_n^{-d_{J^c}} e^{ -2 }
\cdot
   \frac{1}{2}  
%\geq C_{11}  \min_{j\leq K, k\in\mathcal{Y}_J}  \alpha_{jk} \sigma_n^{-d_{J^c}}  N_J 
\notag \\
\label{eq:bd_fratio} 
&\geq 2 C_{3}  \frac{\sigma_n^{2\beta }}{m^2}
= 2 \lambda_n.
\end{align}
By assumption \eqref{eq:ass_y_I_tail_indep_x}, for any $x \in \mathcal{X}$, any $y_J \in \mathcal{Y}_J$, and all sufficiently large $n$,
\begin{equation}
\label{eq:bd_IntYI}
\int_{A_{y_I}}f_{0|J}(\tilde{x} | y_J)\pi_{0J}(y_J)d\tilde{y}_I \leq 2 
\int_{A_{y_I} \cap \{\tilde{y}_I: \, \norm{\tilde{y}_I} \leq a_{\sigma_n} \}} f_{0|J}(\tilde{x} | y_J)\pi_{0J}(y_J)d\tilde{y}_I.
\end{equation}
For any $x \in \mathcal{X}$ with $\norm{{x}} \leq a_{\sigma_n}$ and  $\tilde{y}_I \in A_{y_I} \cap \{\tilde{y}_I: \, \norm{\tilde{y}_I} \leq a_{\sigma_n} \}$,
we have $\norm{\tilde{x}} \leq 2a_{\sigma_n}$ and
\begin{align}
\frac{p(y_J,y_I,{x} | \theta, m)}{p_{0}( y_J,y_I,x)} 
&= 
\frac{\int_{A_{y_I}} f_J(y_J,\tilde{x} | \theta, m)d\tilde{y}_I}{\int_{A_{y_I}}f_{0|J}(\tilde{x} | y_J)\pi_{0J}(y_J)d\tilde{y}_I} \notag
\\
&\geq \frac{\int_{A_{y_I} \cap \{\tilde{y}_I: \, \norm{\tilde{y}_I} \leq a_{\sigma_n} \}} f_J(y_J,\tilde{x} | \theta, m)d\tilde{y}_I}
{2\int_{A_{y_I} \cap \{\tilde{y}_I: \, \norm{\tilde{y}_I} \leq a_{\sigma_n} \}} f_{0|J}(\tilde{x} | y_J)\pi_{0J}(y_J)d\tilde{y}_I} 
\geq \lambda_n,
\label{eq:ratio_p_geq_lambda}
\end{align}
 where the first inequality follows from  \eqref{eq:bd_IntYI} and the second one from \eqref{eq:bd_fratio} combined with Lemma \ref{lm:p_f_distance_bds}.

Next, let us
bound $f_J(y_J,\tilde{x} | \theta, m)/f_{0|J}(\tilde{x} | y_J)\pi_{0}(y_J)$ from below for
$\tilde{x} \in \mathcal{\tilde{X}}$ such that $\norm{{x}} > a_{\sigma_n}$ and $\norm{\tilde{y}_I} \leq a_{\sigma_n}$.
%\begin{align*}
%p(y,x | \theta, m_x,m_y)/f_0(x | y)\pi_0(y) \geq ?
%\end{align*}
For any $j\leq K$ and $k\in\mathcal{Y}_J$, $||\tilde{x}-\mu_{jk,J^c}||^2 \leq 2(||\tilde{x}||^2 + ||\mu_{jk,J^c}||^2) \leq 16||{x}||^2$ as  $||\mu_{jk,J^c}||\leq 2a_{\sigma_n}$ by construction of $U_{j|k}$ and $2||x||>||\tilde{x}||$. Then
\begin{align*}
\phi(\tilde{x},\mu_{jk,J^c},\sigma_{J^c}) &= (2\pi)^{-d_{J^c/2}} \prod_{i\in J^c}\sigma_i^{-1}
 \exp\left\{ 
-0.5 \sum_{i \in J^c}(\tilde{x}_i-\mu_{jk,i})^2/\sigma_i^2
\right\} 
\\
&\geq 
(2\pi)^{-d_{J^c/2}} \sigma_n^{-d_{J^c}} 
\exp\left\{ - \frac{8||{x}||^2}{\underline{\sigma}_{n}^2} \right\}.
\end{align*}
Then, for $n$ large enough
\begin{align*}
f_J(y_J,\tilde{x}|\theta,m)
&=
\sum_{k\in\mathcal{Y}_J} \sum_{j=1}^{K}
\alpha_{jk}\int_{A_{y_J}} \phi(\tilde{y}_J,\mu_{jk,J},\sigma_J) d\tilde{y}_J \phi(\tilde{x},\mu_{jk,J^c},\sigma_{J^c}) \\
&\geq (2\pi)^{-d_{J^c/2}} \sigma_{n}^{-d_{J^c}}
\exp\left\{ - \frac{8||{x}||^2}{\underline{\sigma}_{n}^2} \right\} 
 \sum_{j=1}^{K} \alpha_{jk} \sum_{k\in\mathcal{Y}_J} \int_{A_{y_J}} \phi(\tilde{y}_J,\mu_{jk,J},\sigma_J) d\tilde{y}_J \\
& \geq (2\pi)^{-d_{J^c/2}} \sigma_n^{-d_{J^c}} 
\exp\left\{ - \frac{8||{x}||^2}{\underline{\sigma}_{n}^2} \right\} 
  \frac{1}{2} K \min_{j,k} \alpha_{j k}.
\end{align*}
Combining this inequality with \eqref{eq:dgp_lower_bound}, we  get
\begin{align}
\frac{f_J(y_J,\tilde{x} | \theta, m)}{f_{0|J}(\tilde{x} | y_J)\pi_{0J}(y_J)} 
 & \geq \frac{1}{2} (2\pi)^{-d_{J^c/2}} \bar{f}_0^\inv \sigma_n^{-d_{J^c}} 
 K \frac{\sigma_n^{2\beta + d_{J^c}}}{2m^2}  
 \exp\left\{ - \frac{8||{x}||^2}{\underline{\sigma}_{n}^2} \right\}
\notag
\\
&  \geq  
\exp\left\{ - \frac{8||{x}||^2}{\underline{\sigma}_{n}^2} - C_4 \log n\right\}
\label{eq:ratio_f_geq_exp_x2}
\end{align} 
for sufficiently large $C_4$
because $|\log \left [ K \sigma_n^{2\beta}/m^2 \right ]| \lesssim \log n$.
 
Thus, for $||{x}|| > a_{\sigma_n}$, \eqref{eq:ratio_f_geq_exp_x2} 
and the first inequality in \eqref{eq:ratio_p_geq_lambda}, which holds for any $x \in \mathcal{X}$, deliver
\begin{align}
\label{eq:boundratio_forlarge_x}
\frac{p(y_J,y_I,{x} | \theta, m)}{p_{0}( y_J,y_I,x)} 
\geq 
\exp\left\{ - \frac{8||{x}||^2}{\underline{\sigma}_{n}^2} - C_4 \log n\right\}.
\end{align}
\end{proof}

\begin{lemma}
\label{lm:gKL_inequality}
Under the assumptions and notation of Section \ref{sec:post_rates},
for $\lambda_n<\lambda_0$, where  $\lambda_0$ is defined in Lemma \ref{lm:dH_KL},
\begin{align*}
&E_0 \left( \left[  \log \frac{p_0(y_J,y_I,x)}{p(y_J,y_I,{x}| \theta, m)}\right]^2\right)
%\leq C_{15} \log(1/\lambda_n)^2\sigma_n^{2\beta} 
\leq A \tepsilon_n^2\\
&E_0 \left( \left[  \log \frac{p_0(y_J,y_I,x)}{p(y_J,y_I,{x}| \theta, m)}\right]\right)
%\leq C_{15} \log(1/\lambda_n)\sigma_n^{2\beta} 
\leq A \tepsilon_n^2
\end{align*}
\end{lemma}

\begin{proof}
\begin{align*}
&E_0 \left( \left[  \log \frac{p_0(y_J,y_I,x)}{p(y_J,y_I,{x}| \theta, m)}\right]^2\right)\\ 
&\leq 
d_H^2(p_0(\cdot,\cdot), p(\cdot,\cdot |\theta,m ))
 \left( 12+ 2 \left(\log\frac{1}{\lambda_n}\right)^2\right) 
+ 8 P \left \{ \bigg( \log \frac{p_0(\cdot,\cdot)}{p(\cdot,\cdot| \theta, m)}\bigg)^2 
\indic \left \{ \frac{p(\cdot,\cdot| \theta, m)}{p_0(\cdot,\cdot)} < \lambda_n \right\} 
 \right\} \\ 
& \lesssim \sigma_n^{2\beta} (12 + 2 \log(1/\lambda_n)^2)+ \sigma_n^{2\beta + \epsilon} 
\lesssim  \log(1/\lambda_n)^2\sigma_n^{2\beta}, 
\end{align*}
where first inequality is derived using Lemma \ref{lm:dH_KL} and penultimate inequality is derived using inequalities (\ref{eq:H2_bound_dgp_model}) and (\ref{eq:exp_log_ratio}). Similarly,
\begin{align*}
&E_0 \left( \log \frac{p_0(y_J,y_I,x)}{p(y_J,y_I,{x}| \theta, m)}\right)\\ 
&\leq 
d_H^2(p_0(\cdot,\cdot), p(\cdot,\cdot |\theta,m ))
 \left( 1+ 2 \left(\log\frac{1}{\lambda_n}\right)\right) 
+ 2 P \left \{ \bigg( \log \frac{p_0(\cdot,\cdot)}{p(\cdot,\cdot| \theta, m)}\bigg)
\indic \left \{ \frac{p(\cdot,\cdot| \theta, m)}{p_0(\cdot,\cdot)} < \lambda_n \right\} 
 \right\} \\ 
& \lesssim \sigma_n^{2\beta} (1 + 2 \log(1/\lambda_n))+ \sigma_n^{2\beta + \epsilon} \lesssim \log(1/\lambda_n)\sigma_n^{2\beta}.
\end{align*}

Furthermore, 
\begin{align*}
\log(1/\lambda_n)\sigma_n^{2\beta} 
&\leq  \log(1/\lambda_n)^2\sigma_n^{2\beta} 
%\\ &
= 
\log \left(  \frac{2N_JK^2}{\sigma_n^{2\beta }}\right)^2 \tepsilon_n^{2} 
(\log(\tepsilon_n^\inv))^{-2}\\
&\leq   \left(\frac{\log [2N_J^2 (C_1 \sigma_n^{-d_{J^c}} \{\log ( \tilde{\epsilon}_n^\inv) \}^{d_{J^c}+d_{J^c}/\tau})^2 \sigma_n^{-2\beta }]}{\log(\tepsilon_n^\inv)}\right)^2\tepsilon_n^{2} ,
%\\
%& 
%\leq  C_{15} \left(\frac{\log [C_1^\star N_J \tepsilon_n^{-const}}{\log(\tepsilon_n^\inv)]}\right)^2\tepsilon_n^{2} 
%\leq A \tepsilon_n^2,
\end{align*}
where the term multiplying $\tepsilon_n^{2}$ on the right hand side is bounded by Assumption \ref{as:Nj_o_n_1nu} ($N_J= o(n^{1-\nu})$) and definitions of $\tepsilon_n$ and $\sigma_n$.
\end{proof}

\begin{lemma}
\label{lm:prior_bound}
Under the assumptions and notation of Section \ref{sec:post_rates},
for all sufficiently large $n$,  $s = 1 + 1/\beta + 1/\tau$, and some $C_6>0$
\begin{eqnarray*}
\Pi(m=N_JK, \theta \in S_{\theta^\star} )  \geq 
\exp \left[-C_6 N_J\tilde{\epsilon}_n^{-d_{J^c}/\beta} \{\log (n)\}^{d_{J^c}s + \max\{\tau_1,1,\tau_2/\tau\} }\right] .
\end{eqnarray*}
\end{lemma}

\begin{proof}
First, consider the prior probability of $m=N_JK$. By \eqref{eq:asnPrior_m} for some $C_{61}>0$,
\begin{align}
\label{eq:KL1}
&\Pi(m=N_JK) \propto \exp[-a_{10} N_JK (\log N_JK)^{\tau_1}] \geq \exp [-C_{61}N_J\tilde{\epsilon}_n^{-d_{J^c}/\beta} \{\log (1/ \tilde{\epsilon}_n)\}^{sd_{J^c}}(\log n)^{\tau_1}] \notag \\ 
& \geq \exp [-C_{61}N_J\tilde{\epsilon}_n^{-d_{J^c}/\beta} \{\log (n)\}^{sd_{J^c}+\tau_1}]
\end{align}
as
$N_J=o(n^{1-\nu})$ by \eqref{eq:Npoly_n} and $\tepsilon_n^\inv<n$.

Second, consider the prior on $\{ \alpha_{jk}\}$. 
There exist $(j_0,k_0)$ such that $\alpha_{j_0k_0}^\star \geq \frac{1}{m}$ 
and suppose that $|\alpha_{jk}^\star-\alpha_{jk}|\leq \frac{\sigma_n^{2\beta}}{m^2}$ for all $(j,k)\neq (j_0,k_0)$. Then, 
\begin{align*}
\left| \alpha_{j_0k_0}^\star-\alpha_{j_0k_0} \right| =
\left| \sum_{(jk)\neq(j_0k_0)} \alpha_{jk}^\star-\alpha_{jk} \right| \leq 
(m-1)\frac{\sigma_n^{2\beta}}{m^2}\leq\frac{ \sigma_n^{2\beta}}{m}
\end{align*}
\begin{align*}
\alpha_{j_0k_0} \geq \alpha_{j_0k_0}^\star - \frac{ \sigma_n^{2\beta}}{m}\geq 
\frac{ 1-\sigma_n^{2\beta}}{m} \geq \frac{\sigma_n^{2\beta + d_{J^c}}}{2m^2}.
\end{align*}
%Also, for $n$ large enough and for any $k\in\mathcal{Y}_{J}$ 
%\begin{align*}
%\sum_{j=1}^{K} \alpha_{jk} &= \sum_{j=1}^{K} (\alpha_{jk}-\alpha_{jk}^\star) +\sum_{j=1}^{K} \alpha_{jk}^\star \geq \pi_{0J}(k) -\sum_{j=1}^{K}| \alpha_{jk}-\alpha_{jk}^\star| \\
%&\geq \frac{\underline{\pi}}{N_J} - (K-1)\frac{\sigma_n^{2\beta}}{m^2}-\frac{\sigma_n^{2\beta}}{m}\geq \frac{\underline{\pi} -
 %(K-1)\frac{\sigma_n^{2\beta}}{Km}-\frac{\sigma_n^{2\beta}}{K}}{N_J}
 %\geq \frac{\underline{\pi}/2}{N_J}\geq \frac{\sigma_n^{d_{J^c}}}{N_J}C_{12}^\inv,
%\end{align*}
%where the last two  inequalities hold for $n$ large enough as $\sigma_n\rightarrow 0$. 
Furthermore,
\begin{align*}
\sum_{j=1}^{K}\sum_{k\in\mathcal{Y}_J} |\alpha_{jk}-\alpha_{jk}^\star| \leq (m-1)\frac{\sigma_n^{2\beta}}{m^2} +\frac{\sigma_n^{2\beta}}{m}\leq 2\sigma_n^{2\beta}.
\end{align*}
It then follows that
\begin{align*}
&\Pi\left(
\sum_{j=1}^{K}\sum_{k\in\mathcal{Y}_J} 
|\alpha_{jk}-\alpha_{jk}^\star| \leq  2 \sigma_n^{2\beta}, 
%\min_{k\in\mathcal{Y}_J}  \sum_{j=1}^{K} \alpha_{jk} \geq \frac{\sigma_n^{d_{J^c}}}{N_J}\cdot C_{12}^\inv,
 \min_{j\leq K, k\in\mathcal{Y}_J} \alpha_{jk} \geq
\frac{\sigma_n^{2\beta + d_{J^c}}}{2m^2}
 \right)\\
 &\geq 
 \Pi\left(
  |\alpha_{jk}-\alpha_{jk}^\star| \leq \frac{\sigma_n^{2\beta}}{m^2}, \alpha_{jk} \geq  \frac{\sigma_n^{2\beta}}{2m^2}
  \text{ for } (j,k)\in\{1,\ldots,K\}\times\mathcal{Y}_J \setminus \{(j_0,k_0)\}
 \right)\\
 &\geq 
\exp\left\{-C_{62} N_J K \log(N_J K/\sigma_n^{\beta})\right\},
\end{align*}
where the last inequality is derived in the proof of Lemma 10 in \cite{GhosalVandervaart:07}
for some $C_{62}>0$ (see, also, Lemma 6.1 in \cite{GhosalGhoshVaart:2000}).
%there is a typo in Lemma 10, should be "j=1,...,N-1" in the first equation line of the proof
Note that 
\begin{align}
& K \log(N_J K/\sigma_n^{\beta})
\leq   \tepsilon_n^{-d_{J^c}/\beta} \log(\tepsilon_n^\inv)^{d_{J^c}s} \log(N_J \tepsilon_n^{-d_{J^c}/\beta-1}\log(\tepsilon_n^\inv)^{d_{J^c}s+1})
\notag
\\
&\lesssim 
 \tepsilon_n^{-d_{J^c}/\beta} \log(n)^{d_{J^c}s+1}.
\end{align}

Assumption (\ref{eq:asnPrior_sigma3}) on the prior for $\sigma_{i}$  implies that for $i \in J$
\begin{align}
\prod_{i=1}^{d_J}
 \Pi(\sigma_{i}^{-2}\geq 32 N_i^2 \beta \log \sigma_n^\inv)&\geq 
\prod_{i=1}^{d_J}
 \left(a_6 (64 N_i^2 \beta \log \sigma_n^\inv)^{a_7} \exp\left\{
 -a_9 (64 N_i^2 \beta \log \sigma_n^\inv)^{1/2}\right\}\right)\notag \\
 &\geq \exp\left\{ -C_{63} N_J \log(\sigma_n^\inv)
\right\} \geq \exp\left\{ -C_{64} N_J \log(n)
\right\},
\end{align}
and for $i \in J^c$,
\begin{align}
\prod_{i=1}^{d_{J^c}}
 &\Pi\left(\sigma_{i,n}^{-2}\leq\sigma_{i}^{-2}\leq \sigma_{i,n}^{-2}(1+\sigma_{n}^{2\beta})\right)\geq 
\prod_{i=1}^{d_{J^c}}
 \left(a_6 (\sigma_{i,n}^{-2})^{a_7}\sigma_{n}^{2a_8\beta} \exp\left\{
 -a_9 \sigma_{i,n}^\inv\right\}\right)
\notag\\
 &\geq \prod_{i=1}^{d_{J^c}} \exp\left\{
-C_{65}\sigma_{i,n}^\inv
\right\} = \prod_{i=1}^{d_{J^c}} \exp\left\{
-C_{65}\sigma_{n}^{-\beta/\beta_i}
\right\}\geq
 \exp\left\{
-C_{65}d_{J^c}\sigma_{n}^{-d_{J^c}}
\right\}\notag \\
&\geq \exp\left\{
-C_{66}\tepsilon_{n}^{-d_{J^c}/\beta}\log(n)^{d_{J^c}/\beta}
\right\}.
\end{align}
Assumption (\ref	{eq:asnPrior_mu_lb}{})  on the prior for $\mu_{jk}$ implies
\begin{align}
\prod_{j=1}^{K}\prod_{k\in\mathcal{Y}_J}\prod_{i \in J} \Pi\left( 
\mu_{jk,i} \in
 \left[ k_i - \frac{1}{4N_i}, k_i + \frac{1}{4N_i}\right]
\right)& \geq \left(a_{11} 2^{-d_J}N_J^\inv \exp\left\{-a_{12}\right\}\right)^{N_JK}\notag \\
&\geq 
\exp\left\{
-C_{67} N_JK \log(N_J)
\right\}
\notag \\
&\geq 
\exp\left\{
-C_{68} N_J  \tepsilon_n^{-d_{J^c}/\beta} \log(n)^{d_{J^c}s+1}
\right\} 
\end{align}
and 
\begin{align}
\label{eq:KL6}
\prod_{j=1}^{K}\prod_{k\in\mathcal{Y}_J}
\Pi\left( \mu_{jk,J^c}\in U_{j|k} \right)& \geq \left(a_{11}  \exp\left\{-a_{12}a_{\sigma_n}^{\tau_2}\right\} \min_{j,k}Vol(U_{j|k})\right)^{N_JK}\notag \\
&= \left(a_{11}  \exp\left\{-a_{12}a_{\sigma_n}^{\tau_2}\right\} \sigma_n^{d_{J^c}} 
\tepsilon_n^{2b_1d_{J^c}}\right)^{N_JK}\notag \\
& \geq \exp\left\{
-C_{69} N_J \tepsilon_n^{-d_{J^c}/\beta} \log(n)^{d_{J^c}s+\max\{1,\tau_2/\tau\}}
\right\}.
\end{align}

It follows from \eqref{eq:KL1} - \eqref{eq:KL6},  
that for all sufficiently large $n$  and some $C_6>0$,
\begin{eqnarray*}
\Pi( \mathcal{K}(p_0, \tilde{\epsilon}_n))  \geq \Pi(m=N_JK, \theta \in S_{\theta^\star} )  \geq 
\exp [-C_6 N_J\tilde{\epsilon}_n^{-d_{J^c}/\beta} \{\log (n)\}^{d_{J^c}s + \max\{\tau_1,1,\tau_2/\tau\} }] .
\end{eqnarray*}
\end{proof}

\subsubsection{Sieve Construction and Entropy Bounds} 
\label{sec:app_sieve_entropy}

\begin{lemma} 
\label{th:sieve}
%Analog of Proposition 2 from \cite{ShenTokdarGhosal2013}. 

For $H \in  \mathbb{N}$, $0<\underline{\sigma}<\overline{\sigma}$, and $\overline{\mu}>0$,
let us define a sieve
\begin{eqnarray}\label{eq:sieve}
\mathcal{F} = \{p(y,x|\theta,m): \; m\leq H, \; \mu_j \in [-\overline{\mu}, \overline{\mu}]^{d}, j=1, \ldots,m, \sigma_i \in [\underline{\sigma}, \overline{\sigma}], i=1,\ldots,d\}.
\end{eqnarray}

For $0<\epsilon < 1$ and $\underline{\sigma}\leq 1$, 
\begin{align*}
M_e(\epsilon, \mathcal{F}, d_{TV}) \leq & 
H \cdot \left \lceil \frac{12\overline{\mu}d}{\underline{\sigma}\epsilon} \right \rceil^{Hd}
\cdot \left [ \frac{ 15}{\epsilon} \right ]^{H}
\cdot \left \lceil \frac{ \log (\overline{\sigma}/\underline{\sigma})} {\log (1 + \epsilon/[12d])} \right \rceil^d.
\end{align*} 
For all sufficiently large $H$, 
large $\overline{\sigma}$ and small $\underline{\sigma}$, 
\begin{align*}
\Pi(\mathcal{F}^c)  \leq 
&
H^2 d \exp\{-a_{13} \overline{\mu} ^{\tau_3}\}  +   \exp\{- a_{10} H(\log H)^{\tau_1}  \} 
%+ \exp\{-b_4 \log \overline{\sigma}\} + \exp\{-b_5 \underline{\sigma}^{-2a_3}\}
\\
&
+ d a_1\exp \{-a_2 \underline{\sigma}^{-2 a_3}\} 
+ 
d a_4 \exp\{-2a_5 \log \overline{\sigma}\}
.  
\end{align*}  
\end{lemma}

\begin{proof}
The proof is similar to proofs of related results in 
\cite{norets_pati_2017}, \cite{ShenTokdarGhosal2013}, and \cite{GhosalVaart:01} among others. 

Let us begin with the first claim.  
For a fixed value of $m$, define set $S_{\mu}^m$ to
contain centers of $|S_{\mu}^m| =  \lceil 12\overline{\mu}d/(\underline{\sigma}\epsilon) \rceil$ equal length intervals 
partitioning $[-\overline{\mu}, \overline{\mu}]$. 
Let $S_\alpha^m$ be an $\epsilon/3$-net of $\Delta^{m-1}$ in total variation distance ($\forall \alpha \in \Delta^{m-1}$, $\exists \tilde{\alpha} \in S_\alpha^m$, $d_{TV}(\alpha,\tilde{\alpha})\leq \epsilon/3$).
From Lemma A.4 in \cite{GhosalVaart:01}, the cardinality of $S_\alpha^m$, is bounded as follows
\[
|S_\alpha^m| \leq [15/\epsilon]^{m}.
\]
Define 
$S_{\sigma} =  \{\sigma^{l},l=1, \ldots, \lceil{ \log (\overline{\sigma}/\underline{\sigma})/ (\log (1 + \epsilon/(12d)}\rceil, \sigma^1= \underline{\sigma}, (\sigma^{l+1} - \sigma^l)/\sigma^l = \epsilon/(12d)\}$.  

Let us show that 
\[
S_{\mathcal{F}} =  \{p(y,x|\theta,m): \; m\leq H, 
\, 
\alpha \in S_{\alpha}^m, \,
\sigma_i \in S_{\sigma}, \,
\mu_{ji} \in S_{\mu}^m,\, 
j\leq m, \, i \leq d \}\] 
is an $\epsilon$-net for $\mathcal{F}$ in $d_{TV}$.  For a given $p(\cdot|\theta,m)  \in \mathcal{F}$ with $\sigma^{l_i} \leq \sigma_i \leq \sigma^{l_i+1}$, $i=1,\ldots,d$,  
find   
$\tilde{\alpha} \in S_{\alpha}^m$,
$\tilde{\mu}_{ji} \in S_{\mu^x}^m$,
and $\tilde{\sigma}_i=\sigma_{l_i} \in S_{\sigma}$
such that for all
$j=1, \ldots, m$ and $i=1,\ldots, d$
\begin{equation*}
|\mu_{ji}  - \tilde{\mu}_{ji} | \leq \frac{\underline{\sigma} \epsilon}{12 d},  
\; 
\sum_j |\alpha_j - \tilde{\alpha}_j|  \leq \frac{\epsilon}{3},
\; 
\frac{|\sigma_i - \tilde{\sigma}_i|}{\tilde{\sigma}_i} \leq \frac{\epsilon}{12d}. 
\end{equation*}

By Lemma \ref{lm:p_f_distance_bds},
$d_{TV}(p(\cdot|\theta,m),p(\cdot|\tilde{\theta},m))\leq d_{TV}(f(\cdot|\theta,m),f(\cdot|\tilde{\theta},m))$.
Similarly to the proof of Proposition 3.1 in \cite{NoretsPelenis:11} or 
Theorem 4.1 in \cite{norets_pati_2017},
\begin{align*}
%\label{eq:L1normBd}
& d_{TV}(f(\cdot|\theta,m),f(\cdot|\tilde{\theta},m)) \leq \sum_j |\alpha_j - \tilde{\alpha}_j| +
2 \max_{j=1,\ldots,m}
|| \phi_{\mu_j,\sigma} -\phi_{\tilde{\mu}_j,\tilde{\sigma}}||_1   \\
%\label{eq:DiffPiBd}
& \leq \epsilon/3 + 
4 \sum_{i=1}^{d} \bigg\{\frac{|\mu_{ji} - \tilde{\mu}_{ji}|}{\sigma_i \wedge \tilde{\sigma}_i} 
+ \frac{|\sigma_i - \tilde{\sigma}_i|}{\sigma_i \wedge \tilde{\sigma}_i} \bigg\} \leq \epsilon
.
\end{align*}
This concludes the proof for the covering number.

The proof of the upper bound on $\Pi(\mathcal{F}^c)$ is the same as the corresponding proof of Theorem 4.1 in \cite{norets_pati_2017}, except here the coordinate specific scale parameters and slightly different notation for 
the prior tail condition \eqref{eq:asnPrior_mu_tail_ub} lead to dimension $d$ appearing in front of some of the terms in the bound.

\end{proof}

\begin{lemma} 
\label{lm:sieve_n}

Consider $\epsilon_n=(N_J/n)^{\beta_{J^c}/(2\beta_{J^c} + 1 )} (\log n)^{t_J}$ and 
$\tilde{\epsilon}_n=(N_J/n)^{\beta_{J^c}/(2\beta_{J^c} + 1 )} (\log n)^{\tilde{t}_J}$ 
 with
 $t_J > \tilde{t}_J +  \max \{0, (1- \tau_1)/2\}$ and
 $\tilde{t}_J > t_{J0}$, where $t_{J0}$ is defined in \eqref{eq:t0def}.
Define $\mathcal{F}_n$ as in \eqref{eq:sieve} with $\epsilon = \epsilon_n, H =
n\epsilon_n^2 /(\log  n)$, 
$\underline{\alpha}= e^{-n H}$,
$\underline{\sigma} = 	n^{-1/(2a_3)}$, $\overline{\sigma} = 	e^n$,
and $\overline{\mu}=n^{1/\tau_3}$. 
Then, for some constants 
$c_1,c_3 >0$ and every $c_2 >0$, 
$\mathcal{F}_n$ satisfies \eqref{eq:entropy} and \eqref{eq:sievecomplement} for all large $n$.  \end{lemma}

\begin{proof}
From Lemma \ref{th:sieve}, 
\[
 \log M_e(\epsilon_n,\mathcal{F}_n,  \rho) 
\leq  c_1 H \log n =  c_1 n\epsilon_n^2.  
\]
Also,
\begin{align*}
 \Pi(\mathcal{F}_n^c) 
& \leq  H^2 \exp\{-a_{13} n\}  
    + \exp\{- a_{10} H(\log H)^{\tau_1} \} 
	%+ \exp\{-b_4 n\} + \exp\{-b_5 n\}
\\
&
+ a_1\exp \{-a_2 n\} 
+ 
a_4 \exp\{-2a_5 n\}
.  
\end{align*}
Hence, $\Pi(\mathcal{F}_n^c) \leq e^{-(c_2+4) n \tilde{\epsilon}_n^2}$ for any $c_2$ if 
$\epsilon_n^2 (\log n)^{\tau_1 -1}/\tilde{\epsilon}_n^2 \rightarrow \infty$, which holds for 
$t_J > \tilde{t}_J +  \max \{0, (1- \tau_1)/2\}$.

\end{proof}

\end{document}